\documentclass[11pt, reqno, psamsfonts]{amsart}

\pdfoutput=1
\usepackage{multirow}
\usepackage{tcolorbox}
\usepackage{graphics}
\usepackage[normalem]{ulem}
\usepackage{comment}

\usepackage{amssymb}
\usepackage{amsthm}
\usepackage{amsmath}
\usepackage{amsfonts}
\usepackage{dsfont}
\usepackage{latexsym}
\usepackage[T1]{fontenc}
\usepackage[utf8]{inputenc}
\usepackage[russian, french, english]{babel}

\usepackage{graphicx}
\usepackage{wrapfig, subcaption}
\usepackage[justification=centering, labelfont=bf]{caption}
\usepackage{mathtools}
\usepackage{tikz}
\usepackage[all]{xy}
\usepackage{amsbsy}
\usepackage{enumitem}
\usepackage{setspace}
\usepackage{mathrsfs}
\usetikzlibrary{shapes,snakes}
\usetikzlibrary{arrows.meta}
\usetikzlibrary{decorations.pathmorphing}
\usetikzlibrary{patterns}
\usepackage{float}
\usepackage{array}
\usepackage{multicol}
\usepackage{stmaryrd}
\usepackage{cancel} 
\usepackage{lmodern}
\usepackage{algpseudocode}
\usepackage{algorithm}
\usepackage{upgreek}
\usepackage{titlesec}
\usepackage[spacing=true,kerning=true,babel=true,tracking=true]{microtype}
\usepackage[shortcuts]{extdash}
\usepackage[foot]{amsaddr}
\usepackage[left=1in,right=1in,top=0.95in,bottom=0.95in,bindingoffset=0cm]{geometry}
\usepackage{centernot}
\usepackage{mdframed}
\usepackage{multirow}
\usepackage{hyperref}
\usepackage{soul}
\hypersetup{
    colorlinks=true,
    linkcolor=blue,
    filecolor=magenta,
    urlcolor=blue,
    citecolor=gray,
}
\allowdisplaybreaks

\usepackage{color}

\makeatletter
\let\orgdescriptionlabel\descriptionlabel
\renewcommand*{\descriptionlabel}[1]{%
  \let\orglabel\label
  \let\label\@gobble
  \phantomsection
  \edef\@currentlabel{#1}%
  \let\label\orglabel
  \orgdescriptionlabel{#1}%
}
\makeatother



\title{Sampling Colorings with Fixed Color Class Sizes}
\date{}
\author{\lsstyle Aiya~Kuchukova}
\email{aiya\_kuchukova@gatech.edu}
\author{\lsstyle Will Perkins}
\email{math@willperkins.org}
\author{\lsstyle Xavier~Povill}
\email{xavier.povill@upc.edu}
\address{\textls{\normalfont{}School of Mathematics, Georgia Institute of Technology, Atlanta, GA, USA}}
\address{\textls{\normalfont{}Universitat Politècnica de Catalunya, Barcelona, Spain}}
\thanks{}

\newtheoremstyle{bfnote}%
{}{}%
{\slshape}{}%
{\bfseries}{\bfseries.}%
{ }%
{\thmname{#1}\thmnumber{ #2}\thmnote{ \normalfont{}#3}}

\newtheoremstyle{Claim}%
{}{}%
{\slshape}{}%
{\itshape}{.}%
{ }%
{\thmname{#1}\thmnumber{ #2}\thmnote{ \normalfont{}#3}}

\usepackage{aliascnt}
\newcommand{\newaliastheorem}[2]{%
  \newaliascnt{#1}{theorem}%
  \newtheorem{#1}[#1]{#2}%
  \aliascntresetthe{#1}%
  \expandafter\providecommand\csname #1autorefname\endcsname{#2}%
}

\theoremstyle{bfnote}
\newtheorem{theorem}{Theorem}[section]
\newtheorem*{theorem*}{Theorem}

\newaliastheorem{lemma}{Lemma}
\newaliastheorem{corollary}{Corollary}
\newaliastheorem{proposition}{Proposition}
\newaliastheorem{conjecture}{Conjecture}
\newaliastheorem{fact}{Fact}
\newaliastheorem{observation}{Observation}
\newaliastheorem{consequence}{Consequence}
\newaliastheorem{remark}{Remark}

\newtheorem*{corollary*}{Corollary}

\theoremstyle{definition}
\newaliastheorem{definition}{Definition}

\newtheorem{claim}{Claim}[theorem]

\newtheorem*{example*}{Example}

\makeatletter
\newcommand{\neutralize}[1]{\expandafter\let\csname c@#1\endcsname\count@}
\makeatother

\newenvironment{claimproof}{\noindent$\rhd$\hspace{1em}}{\hfill$\blacktriangleleft$\smallskip}

\newcommand{\R}{\mathbb{R}}
\renewcommand{\C}{\mathbb{C}}

\renewcommand{\P}{\mathbb{P}}
\newcommand{\E}{\mathbb{E}}
\renewcommand{\epsilon}{\varepsilon}
\newcommand{\eps}{\epsilon}
\renewcommand{\phi}{\varphi}

\newcommand{\vl}{\vec{\lambda}}

\newcommand{\abs}[1]{\left\vert#1\right\vert}
\DeclareMathOperator*{\argmax}{arg\,max}

\newcommand{\defeq}{:=}

\DeclareMathOperator{\cov}{Cov}
\DeclareMathOperator{\Cov}{Cov}
\DeclareMathOperator{\var}{Var}
\DeclareMathOperator{\Var}{Var}

\renewcommand{\Re}{\operatorname{Re}}
\renewcommand{\Im}{\operatorname{Im}}

\usepackage{todonotes}

\definecolor{darkpastelgreen}{rgb}{0.01, 0.75, 0.24}

\titleformat{\section}[block]{\scshape\filcenter}{\thesection.}{1ex}{}
\titleformat{\subsection}[block]{\bfseries}{\thesubsection.}{1ex}{}
\titleformat{\subsection}[block]{\bfseries}{\thesubsection.}{1ex}{}
\titleformat{\subsubsection}[runin]{\itshape}{\bfseries\upshape\thesubsubsection.}{1ex}{}[.---]

\titlespacing*{\section}{0pt}{*3}{*1}
\titlespacing*{\subsection}{0pt}{*3}{*1}
\titlespacing*{\subsubsection}{0pt}{*1.5}{*0}

\setlist{topsep=4pt,itemsep=4pt}

\usepackage{thmtools}
\usepackage{thm-restate}

\begin{document}

\maketitle

\begin{abstract}
In 1970 Hajnal and Szemer\'edi proved a conjecture of Erd\"os that for a graph with maximum degree $\Delta$, there exists an equitable $\Delta+1$ coloring; that is a coloring where color class sizes differ by at most $1$. 
In 2007 Kierstand and Kostochka  reproved their result and provided a polynomial-time algorithm which produces such a coloring.
In this paper we study the problem  
of approximately sampling uniformly random equitable colorings. 
A series of works gives polynomial-time sampling algorithms for colorings without the color class constraint, the latest improvement being by Carlson and Vigoda for $q\geq 1.809 \Delta$. 
In this paper we give a polynomial-time sampling algorithm for equitable colorings when $q> 2\Delta$. Moreover, our results extend to colorings with small deviations from equitable (and as a corollary,  establishing their existence). 
The proof uses the framework of the geometry of polynomials for multivariate polynomials, and as a consequence establishes a multivariate local Central Limit Theorem for color class sizes of uniform random colorings. 
\end{abstract}

\section{Introduction}
\subsection{Sampling with Constraints}
Many interesting problems in combinatorics and its applications can be formulated as problems about independent sets, matchings, graph cuts, or colorings with global constraints. Thus, the ability to sample those fundamental objects with  constraints solves a broad range of nuanced questions and strengthens current techniques and tools used for sampling. We begin by providing examples of some fundamental objects that we can sample and how introducing constraints alters the sampling problem.

One example comes from the ferromagnetic Ising model, a statistical physics model of graph cuts and magnets. It is possible to approximately count \cite{jerrum1993polynomial} and sample the states of the model \cite{randall1999sampling} for all parameters and all graphs. However, when restricting the magnetization (number of vertices on one side of the cut/ number of particles with a positive spin), an algorithmic thresholds emerges. 
For the ferromagnetic Ising model with fixed magnetization, Carlson, Davies, Kolla, and Perkins established a computational threshold \cite{carlson2022computational}, meaning that there are regimes of temperature and magnetization for which efficient sampling algorithms exist and regimes in which no sampling is possible unless NP=RP  
(example of dynamical threshold also shown in \cite{kuchukova2025fast}).
Another example of interest comes from \textit{independent sets}, often studied using the hard-core model. When introducing a size constraint on the independent sets, 
Davies and Perkins \cite{davies2023approximately} showed the existence of a  computational threshold;  there exists $\alpha_c$ such that it is possible to sample independent sets of size $\alpha n$ for $\alpha<\alpha_c$ in polynomial time, but no sampling algorithm exists for $\alpha > \alpha_c$ unless NP=RP. Work of Jain, Michelen, Pham, and Vuong \cite{jain2023optimal}, establishes fast mixing of the down-up walk on independent sets of size $k \leq (1 - \delta)\alpha_c$.
 
Placing constraints on the size of the object is natural for statistical physics models, since they differentiate between settings where the number of particles can vary, called \textit{grand canonical ensembles}, versus settings where the number of particles is fixed, called \textit{canonical ensembles}. In practice, there are a lot more results in sampling algorithms for the grand canonical ensemble, since placing an additional constraint adds a layer of complexity. Under additional constraints the behavior of the system can change completely; for example, a well behaved system (like the ferromagnetic Ising model on a random graph) can exhibit `glassy' behavior with the imposition of a constraint on the magnetization~\cite{mezard1987mean}.

Observe that the examples of constraints we have mentioned so far are one-dimensional. In this paper we will explore the instance of placing multiple global constraints on the model. More precisely, we will show that given an integer vector (with some assumptions on it), we can sample colorings in which color class sizes form the requested vector. It is a delicate problem since each of the color class sizes is an independent set of restricted size, and moreover, the union of colors should cover all the vertices of the graph. We introduce the background about sampling colorings as well as known existence results in the next subsection.

\subsection{Sampling Colorings with Multiple Constraints}
Before looking at colorings with constraints, we give background for sampling colorings with no constraints, since this problem is also far from being well understood. The following is a classical conjecture about sampling colorings 
\cite{jerrum1995very, vigoda1999improved}:
\begin{conjecture}
    For the class of graphs of maximum  degree $\Delta$ and $q \ge \Delta +2$, there is a polynomial-time algorithm to approximately sample uniformly random $q$-colorings.
\end{conjecture}

Moreover, it is conjectured that a simple Markov chain, the Glauber dynamics, achieves this.

There has been a lot of progress in the past 30 years towards the resolution of the conjecture. For example, in his seminal paper, Jerrum proved it for $q \geq 2\Delta + 1$ \cite{jerrum1995very}. Another exciting result was by Vigoda, which samples colorings for $q > \frac{11}{6}\Delta$ \cite{vigoda1999improved}. There were also a lot of impressive improvements on the bounds for special families of graphs 
\cite{dyer2003randomly, molloy2002glauber, hayes2005general, martinelli2007fast, efthymiou2018sampling, mossel2010gibbs, hayes2003non, sly2017glauber,chen2023strong}. 
For general graphs, $q$ was further improved to $(11/6 - \eps) \Delta$ by Chen, Delcourt, Moitra, Perarnau, and Postle \cite{chen2019improved}, and recently to $q \geq 1.809\Delta$ by Carlson and Vigoda \cite{carlson2025flip}.

Despite the impressive progress in this direction, the problem presents a considerable challenge and might require many more ideas to decrease the necessary amount of colors down to $\Delta +2$. Thus, it seems important to ask the following questions: How refined are our current tools and techniques? How well can we control the colorings that we can sample? 

We now define the size constraints on color classes we have previously mentioned. 
\begin{definition}[$\vec{n}$-coloring]
    Given a vector $\vec{n} = (n_1, \dots, n_q)$, an $\vec{n}$-coloring is a proper coloring of the vertices of a graph such that there are $n_i$ vertices of color $i$, $\forall i \in [q]$.  
\end{definition}

Let us describe perhaps the simplest case of $\vec{n}$ colorings, equitable colorings. 

\begin{definition}[Equitable colorings]
    An \textit{equitable} $q$-coloring is a proper coloring with $q$ colors such that the sizes of every pair of color classes  differ by at most 1. Note that if $n$ divides $q$, the color class sizes are equal.
\end{definition}

Erd\H{o}s conjectured in 1964 that  any graph of maximum degree $\Delta$ has an equitable $(\Delta + 1)$-coloring. In 1970 Hajnal and Szemer\'edi gave a proof of the conjecture \cite{hajnal1970proof}. Later Kierstead and Kostochka provided a simpler proof, which also gives a polynomial-time algorithm for constructing such a coloring. \cite{kierstead2008short}. 

As a corollary of our main theorem, we establish a polynomial-time algorithm for approximately sampling uniform equitable colorings when $q \ge 2 \Delta$. 

\begin{definition}
Let $\mu$ be the target distribution. An algorithm is said to $\eps$-approximately sample from $\mu$ if the distribution of its output $\nu$ satisfies $\|\nu - \mu \|_{\text{TV}} \leq \eps$.   
\end{definition}

\begin{restatable}[Sampling Equitable Colorings]{theorem}{runtimebalanced}
\label{thm:runtime balanced}
There exists a sampling algorithm that, given $q\geq 2\Delta$ and any $G$ on $n$ vertices from the class of graphs with maximum degree $\Delta$, $\eps$-approximately samples equitable colorings on $G$ with high probability with running time $O\!\left(n^{(q+1)/2}\, \,\log n \,\log(\frac{1}{\eps})^2\right)$. 
\end{restatable}

In a more general version of the theorem, we show the existence of a sampling algorithm for colorings that we call ``skewed'', establishing their existence as a corollary. 
\begin{theorem}[Sampling Skewed Colorings]
    Fix $q\geq 2\Delta +1$. There exists constant $c = c(\Delta)>0$ small enough, such that given any $G$ on $n$ vertices from the class of graphs with maximum degree $\Delta$ and $\vec{n}$ satisfying $\| \vec n \|_1 = n$ and  $|n_i - \frac{n}{q}|<c n$ for all $i\in[q]$, there is an algorithm that $\eps$-approximately samples $\vec{n}$-colorings on $G$ with high probability with running time $O\!\left(n^{q}\, \,\log n \,\log(\frac{1}{\eps})^2\right)$. 
\end{theorem}

As a simple corollary we obtain the existence of such colorings. 

\begin{corollary}
Fix $q \geq 2\Delta +1$. Then there exists constant $c>0$ small enough such that for $n$ large enough, and any $\vec{n}$ satisfying $\| \vec{n} \|_1 = n$ and $| n_i- \frac{n}{q} | \le c n$ for all $i\in [q]$,  every graph $G$ of maximum degree  $\Delta$ on $n$ vertices has an $\vec{n}$-coloring.
\end{corollary}

To the best of our knowledge, very little is known about the existence of skewed colorings beyond the equitable case.
We conjecture that a bound on each color class size (matching the size of an independent set guaranteed to exist in a graph of max degree $\Delta$) suffices. This conjecture generalizes the theorem of Hajnal and Szemer\'edi.

\begin{conjecture}
    For any graph $G$ with maximum degree $\Delta$ and $\vec{n}$ satisfying $\| \vec{n} \|_1 = n$ and $n_i \leq \left\lfloor \dfrac{n}{\Delta + 1} \right\rfloor$ for all $i \in [q]$, there exists an $\vec{n}$-coloring of $G$. 
\end{conjecture}

Note that the conjecture is false even if we ask for $n_i \leq \left\lceil \dfrac{n}{\Delta + 1} \right\rceil$ (a union of cliques of size $\Delta+1$ with $r$ independent vertices, for $1 \leq r \leq \Delta - 1$, is a counterexample). The conjecture is tight for a union of cliques.  

Since $q \geq 2\Delta$ is a natural barrier, a first step in proving this conjecture might be to prove existence on the stronger condition on color class sizes of $n_i \leq \left\lfloor \dfrac{n}{2\Delta } \right\rfloor$.

We also conjecture that it should be possible to sample colorings that are much more skewed, although the proof would require new ideas.

\begin{conjecture}
    There is a polynomial-time approximate sampling algorithm for uniform $\vec n$-colorings in graphs of maximum degree $\Delta$, for $\vec n$ satisfying $n_i \leq \left\lfloor \dfrac{n}{2\Delta } \right\rfloor$ for all $i \in [q]$
\end{conjecture}

\subsection{Short Preliminaries: Potts and Coloring Models, Univariate Zero-Freeness}
Let us give definitions of some statistical physics models that we use throughout the  paper.

The following model is called anti-ferromagnetic Potts model, it is a generalization of another popular statistical physics model called anti-ferromagnetic Ising model. It differentiates between improper and proper colorings by introducing weights that penalize improper colorings.  
\begin{definition}[Anti-ferromagnetic Potts model]\label{definition: potts penalty}
    Let $w$ be the penalty parameter. We define \textit{the partition function} \[Z_G(w) \defeq \sum_{\sigma \in [q]^V} w^{|m_G(\sigma)|},\] where $m(\sigma)$ is the number of monochromatic edges in the coloring $\sigma$ (not necessarily proper). 
    
    This also defines a natural Gibbs measure: if $X$ is the random (not necessarily proper) coloring, then
    \[ \P_{G,w}[X = \sigma  ] = \frac{w^{m_G(\sigma)}}{Z_G(w)} .\]
\end{definition}

Note that at $w = 0$ the partition function counts the number of proper colorings, and the Gibbs measure is simply a uniform distribution over proper colorings.  
If the number of colors satisfies $q \geq \Delta + 1$, a proper coloring exists, and thus the partition function at $w=0$ is non-zero.
Even though the probability distribution requires $w$ to be non-negative and real, it is possible to work with the partition function for complex $w$. 
Liu, Sinclair, and Srivastava proved that in the complex plane around $w=0$ there exists a zero-free region
\cite{liu2025correlation}.
\begin{theorem}[\cite{liu2025correlation}]
    There exists a $\tau_{\Delta} > 0$ such that the following is true. Let $D_{\Delta}$ be a simply connected region in the complex plane obtained as the union of disks of radius $\tau_{\Delta}$ centered at all points on the segment $[0,1]$. For any graph $G$ of maximum degree at most $\Delta \geq 3$ and integer $q \geq 2\Delta$, $Z_{G,q}(w)\not = 0$ when $w \in D_{\Delta}$. 
\end{theorem}
Using their zero-freeness result and a theorem by Barvinok \cite{barvinok2016combinatorics}, they get a deterministic algorithm for approximately counting proper colorings of a graph $G$ with bounded maximum degree and a sampling algorithm for colorings of $G$. We will explain more on how zero-freeness results can be useful for combinatorial algorithms in the next section.

 We will use a variant of the Potts model with external fields (weights) for each of the colors. Since the model is only defined over proper colorings, we will, for simplicity, refer to it as the coloring model (equivalently it can be defined over all colorings with monochormatic edges having penalty $w=0$). 
\begin{definition}[Coloring model with external fields]\label{def:color_weighted_potts_model}
    The \textit{coloring model} is the model defined through the partition function
    \[Z_G(\lambda_1, \dots, \lambda_q) \defeq \sum_{\substack{\sigma \in [q]^V \\ \sigma \text{ proper}}}  \prod_{i = 1}^q \lambda_i^{|\sigma^{-1}(i)|}\]
    The vector $\vec \lambda \in \mathbb C^q$ is called the \textit{fugacity}, and represents the weight assigned to each of the colors. Observe that for $\vec \lambda = \vec 1$, the partition function simply counts the number of proper colorings. For positive real $\vec \lambda$, the model induces a Gibbs measure on the space of proper colorings, given by
    \[ \P_{G,\vl}(\sigma) \defeq \frac{\prod_{i = 1}^q \lambda_i^{|\sigma^{-1}(i)|}}{Z_G(\lambda_1, \dots, \lambda_q)}. \]
\end{definition}
We will often drop the subscripts of $\P_{G, \vl}$ where there is no confusion. 

As in the previous model, probability distribution requires $\vl$ to be a non-negative real vector, however, we can still analyze the partition function for complex $\vl$, which is explained in the subsection below.

\subsection{Multidimensional Zero-Freeness}
An extremely powerful property of the model in a certain range of regimes is the absence of complex zeros (\textit{zero-freeness}) of the partition function. It can imply properties of the model, like deterministic counting \cite{yao2022polynomial},  Spectral Independence \cite{chen2024spectral}, Central Limit Theorems \cite{michelen2024central}, Strong Spatial Mixing \cite{regts2023absence}, and cluster expansion \cite{scott2005repulsive}, and thus has applications in approximation and sampling algorithms. 

There is a lot of existing work that establishes zero-freeness for different models, and curiously, with different motivations. For example, Lee-Yang theorem shows that, for certain partition functions, their zeros lie on the complex unit circle for any graph \cite{yang1952statistical}. One of the theorem's implications is the absence of phase transitions (absence of drastic changes in the behavior of the model) almost everywhere in the complex plane for the ferromagnetic Ising model. See also more work on the hard-core model \cite{peters2019conjecture}, the Ising model \cite{peters2020location}, \cite{patel2023near}, and the monomer-dimer model \cite{heilmann1972theory}. It is important to note that in most cases one has to assume that a graph has bounded maximum degree $\Delta$, since the zero-freeness radius often depends on $\Delta$. 

Our motivation for proving zero-freeness is getting very precise control over the cumulants (expectation, covariances and more) of the color class sizes. 
Derivatives of $\log Z$ are tightly connected with these cumulants, and thus showing analyticity of $\log Z$ (which is directly implied by zero-freeness of $Z$) lets us control the cumulants. 

For example, zero-freeness implies linear upper bounds on the cumulants of any order. As one of the consequences, we can compute the asymptotics of the determinant of the covariance matrix, which will be important for the runtime of our algorithm. Moreover, analyticity guarantees the convergence of the Taylor expansion, which we apply to the characteristic function of the random vector of color class sizes. This helps us establish a Local Central Limit Theorem (LCLT), a result that guarantees that the Gibbs distribution approaches a normal distribution (in the limit) pointwise. 
We believe that this fact is interesting combinatorially in its own merit. Moreover, we use it for the rejection sampling part of our algorithm, see an overview in Subsection \ref{subsection: exposition CLT} and details in Section \ref{section: LCLT}.  
 
Let us now give an overview of the zero-freeness result for the coloring model (see Definition \ref{def:color_weighted_potts_model}). The approach we are taking is similar to the blueprint developed by Liu, Sinclair, and Srivastava for the Potts model with penalty \cite{liu2025correlation} (see Definition \ref{definition: potts penalty}), however, we need to generalize their result to a multivariate partition function and handle $q$ parameters (one for each color) instead of one. This is done so that each color receives its own variable, which lets us regulate each color class size. Generalizing existing tools to work in multivariate settings gives control over different parameters, which allows for better understanding and efficient algorithms for models with global constraints.

\begin{definition}[Assumptions]\label{intro:assumptions}
Let $\nu = 0.9$, $\eps_I \defeq 10^{-4} \Delta^{-4} $. From now on, we will assume that $\vec{\lambda} \in \C^{q}$ is chosen so that for any $i, j \in [q]$,
    \begin{enumerate}
        \item  $\left \vert \arg \lambda_i \right \vert \leq \nu \varepsilon_I / 2$; \label{assumption:lambda-argument-small}
        \item $|\Re \ln \frac{\lambda_i}{\lambda_j}| \leq \Delta \eps_I$; \label{assumption:lambdas-close-together} 
        \item $\lambda_i \notin (-\infty, 0]$\label{assumption:lambdas-nonzero};
    \end{enumerate}
\end{definition}

\begin{theorem}[(Zero-freeness around $\vec{1}$)]\label{intro-thm:zero-freeness-region}
    For $q \geq 2\Delta$, $\Delta \geq 1$ and graph $G$ with maximum degree $\Delta$, and $\vl \in \C^{q}$ such that the assumptions from Definition \ref{intro:assumptions} hold, then $Z_G(\vl) \neq 0$. 
\end{theorem}

The corollary of this result lets us rephrase it in terms of the zero-freeness ball with constant $l_\infty$-radius:

\begin{corollary}\label{intro-corollary:zero-freeness-radius}
    Let $R \in \mathbb R$ such that $0 < R \leq \nu \varepsilon_I / 4 \approx 2.2 \times 10^{-5}  \Delta^{-4}$. If $\vec \lambda \in \mathbb C^q$ is chosen so that $\vert \lambda_i - 1 \vert \leq R$ for every $i \in [q]$, then the assumptions from Definition \ref{intro:assumptions} are satisfied. Hence, for $q \geq 2\Delta$, $\Delta \geq 1$,  and graph $G$ with maximum degree $\Delta$,  $Z_G(\vl) \neq 0$ in a polydisc of radius $R$ around $\vec{1}$.
\end{corollary}

The proof can be found in Section \ref{section: zero-freeness}. The main idea behind the proof is the induction argument in terms of the number of unpinned (not yet colored) vertices. For each iteration of the induction, we show that the ratio of the probability of a vertex to be colored $i$ vs be colored $j$ is very close to 1.

We note that in the paper of Liu, Sinclair, and Srivastava \cite{liu2025correlation} the theorem holds generally for colorings they call ``admissible'' and requires fewer colors for triangle-free graphs. It is also mentioned that their result should generalize to list colorings. We expect these generalizations to hold for our model as well. 
Note that we write theorems in a general enough way that an improvement to the zero-freeness region would lead to an immediate improvement on the sampling and existence result. 

\subsection{A Local Central Limit Theorem}\label{subsection: exposition CLT}
As one of our main theorems, we prove a Local Central Limit Theorem (LCLT), which states that the Gibbs distribution of the coloring model approaches a normal distribution pointwise. We also get asymptotic control over the eigenvalues and determinant of the covariance matrix, and thus can lowerbound the probability of sampling an $\vec{n}$-coloring. This is the main ingredient that goes into the rejection sampling.

\begin{restatable}[(Local Central Limit Theorem)]{theorem}{lclt}
\label{thm:lclt}
Let $q\geq 2\Delta$, $\Delta \geq 1$, and $\vl \in \R^q$ such that $\|\vl - \vec{1}\|_\infty\leq R/2$, where $R$ is a zero-freeness radius. 
Let $\vec X  = (X_1, \dots, X_{q-1})$ be the random vector of color class sizes with the last color dropped. We will denote its expected value and covariance matrix as $\mu := \E[\vec X]$ and $\Sigma := \cov(\vec X)$. 
Then, for $\vec{n} \in \mathbb{Z}_{\geq 0}^{q-1}$ such that $\sum n_i \leq n$,
\[\mathbb P_{\vec{\lambda}}(\vec{X}=\vec{n})
= (1+o(1))\frac{1}{(2\pi)^{(q-1)/2}\sqrt{\det\Sigma}}
\exp\left( -\tfrac12 (\vec{n}-\vec{\mu})^\top\Sigma^{-1}(\vec{n}-\vec{\mu}) \right)
+
o\!\left(n^{-(q-1)/2}\right).
\]
\end{restatable}

\begin{restatable}[(Covariance Determinant Asymptotics)]{lemma}{detasymp}\label{det asymp}
    Let $q \geq \max{\{2\Delta, 3\}}$, $\Delta \geq 1$, and $\vl \in \C^q$ such that $|\vl - \vec{1}|\leq R/2$, where $R$ is a zero-freeness radius,
    \[\det(\Sigma) = \Theta(n^{q-1}).\]
\end{restatable}

Note that if we choose $\vec{\lambda} = \vec{1}$ (corresponding to uniform distribution), by symmetry $\E[X_i] = n/q$. We show that if  $\|\vec{n} - \vec{\mu}\|<C\sqrt{n}$, 
$\exp\left(
-\tfrac12 (\vec{n}-\vec{\mu})^\top\Sigma^{-1}(\vec{n}-\vec{\mu})\right) = \Omega(1)$, and so
\[\pi(\vec{X}=\vec{n})=\Theta\!\left(n^{-(q-1)/2}\right).\]

This is an important component of the proof, since it provides a guarantee for the number of times a rejection sampling step has to be repeated to reach the desired coloring. We explain how that is done in the next subsection. 

The proof of LCLT uses two simple ingredients. The first one is the Taylor expansion of the logarithm of the characteristic function ($\phi \defeq \E[e^{i\langle t,X\rangle}]$) , which can be done in the region where the characteristic function is non-zero (this region is contained in the original zero-freeness region of $Z(\vl)$). It is easy to show that for small $t$, $\phi$ is dominated by the terms corresponding to expectation and covariances (see Subsection \ref{subsection: cluster of char}) For larger $t$ we utilize a lemma which shows an upper bound on characteristic function, thus showing that the contribution of larger $t$ is negligible (see Subsection \ref{subsection: upperbound char}).  

\subsection{Sampling Equitable Colorings}
In this subsection we will explain how to sample equitable colorings and then in the next subsection we explain how to adjust the algorithm slightly to sample skewed colorings.  
We use LCLT as the main ingredient to give us guarantees on the success of rejection sampling. This is a standard tool in sampling (see \cite{jerrum1986random} for example).

By a classical result of Vigoda \cite{vigoda1999improved}, there exists an algorithm that approximately samples colorings (uniformly) in $O(n \,\log n \, \log(1/\eps))$ for $q\geq \frac{11}{6}\Delta$. One can use Glauber dynamics as an example of such algorithm  
(see \cite{blanca2022mixing}, \cite{liu2021coupling} for relevant details on mixing time). 
We use this as the first step in the rejection sampling algorithm.

\begin{restatable}[Rejection Sampling Algorithm]{definition}{rejalg}
\label{def:rej alg}
\, 
    \begin{enumerate}
        \item Approximately sample a uniform coloring (for example by running Glauber dynamics for $T = O_{\eps}(n\log n)$ steps). 
        \item Check if the coloring produced is an $\vec{n}$-coloring. If yes, output the coloring. If no, reject and repeat the first step.
    \end{enumerate}
\end{restatable}

Since LCLT together with bounds on the determinant of the covariance matrix imply that $\pi(\vec{X}=\vec x)=\Theta\!\left(n^{-(q-1)/2}\right)$, it is easy to show that after $O(n^{(q-1)/2} \log(1/\eps))$ iterations of rejection sampling, the probability of the failure of the algorithm is at most $1-\eps$. Hence, the total running time of the algorithm is $O\!\left(n^{(q+1)/2}\, \log(n)\,\log(\frac{1}{\eps})^2\right)$. 
Note that this algorithm actually lets us sample more than just equitable colorings. For any constant $c > 0$ and $\vec n$ such that $\sum_i n_i = n$ and $|n_i - n/q| < c \sqrt{n}$ we know that LCLT and bounds on eigenvalues of the covariance matrix imply that $\pi(\vec{X}=\vec x)=\Theta\!\left(n^{-(q-1)/2}\right)$, and hence the same rejection algorithm works.

\subsection{Sampling Skewed Colorings}\label{subsection: exposition skewed sampling}
It is natural to predict that one could sample skewed colorings ($|n_i - n/q|< cn$ for small $c$) the same way that we sample equitable colorings. However, there are several challenges that appear. One of such challenges is showing that we can find (algorithmically) weights $\vl$ such that they give the right (or close to the right) color class sizes in expectation. For the equitable colorings (and colorings with $O(\sqrt{n})$ deviation from equitable) it was enough to take $\vl = \vec{1}$ (to sample from uniform distribution). To able to sample colorings with a linear deviation from equitable, we need to find $\vl$ such that $\|\E_{\vl}[\vec X] - \vec n\|_{\infty} \leq c\sqrt{n}$. This will allows us sample colorings from Gibbs distribution corresponding to $\vl$ and use LCLT to show that rejection sampling works. Note that the probability distribution we care about only makes sense in the zero-freeness region, thus it is not enough to only find appropriate choice of weights, but also show that they reside within the zero-freeness region. 

To find $\vl$, we discretize our zero-freeness region such that $\lambda_i = 1 + k_i \frac{1}{\sqrt{n}}$ for $i \in [q-1], k_i \in \{-\lfloor R \sqrt n \rfloor, \dots, \lfloor R \sqrt n \rfloor\}$ (we keep $\lambda_q = 1$ for simplicity). We show that the expectation map is $O(n)$-Lipschitz (Subsection \ref{subsection: Lipschitzness of the expectation map}) and since the distance between optimal $\vl$ and one of the candidate $\vl'$ is at most $\frac{1}{\sqrt{n}}$ (in $l_2$), the difference between expectations is at most $\sqrt{n}$ for each color. Hence, we simply run the rejection sampling for each of the candidate $\vl$ until we obtain the $\vec{n}$ colorings. Since there are $\Theta(n^{(q-1)/2})$ candidates, the runtime is $O\!\left(n^q\, \log(n)\,\log(\frac{1}{\eps})^2\right)$. 

To understand what colorings we can sample with this technique, we show that the expectation map (restricted to a certain hyperplane) is surjective on the zero-freeness region (Subsection \ref{subsection: Surjectiveness of the expectation map}), which let's us show that this procedure works for any coloring such that $\sum n_i = n$ and $\|\vec{n} - \frac{n}{q} \vec{1}\|_{\infty} \leq cn$ for $c$ small enough. 

We also need to make sure that we can sample from the Gibbs distribution corresponding to $\vl$. This can be done by Glauber dynamics in $O(n\log n)$ steps. We prove the mixing time for $q\geq 2\Delta +1$ using a standard path coupling technique in Subsection \ref{subsection: Glauber}.

\subsection{Organization}
The paper has the following structure: Section \ref{section: zero-freeness} proves the zero-freeness result for the coloring model, Section \ref{section: LCLT} proves Local Central Limit Theorems for the vector of color class sizes in a random coloring and conclude with the runtime of the rejection sampling algorithm, and Section \ref{section: skewed colorings} shows how to find $\vl$ which give the right expected color class sizes and shows the runtime of Glauber given $\vl$ in the zero-freeness region. 
The Figure \ref{fig: Proof overview} shows the structure of the proof. Note that for sampling equitable colorings only the middle layer of the figure is necessary. 
\begin{figure}[!htbp]
\centering
$$\xymatrix{
 && \text{Existence of appropriate $\vl$}  \ar[drr]  \\
  \text{Zero-freeness} \ar[rr] \ar[urr] && \text{Local Central Limit Theorem}\ar[rr]  && \text{Rejection sampling} \\
  \text{Path coupling} \ar[rr]
 &&\text{Fast mixing of Glauber dynamics} \ar[urr]&&
} $$
\caption{Proof outline}
\label{fig: Proof overview}
\end{figure}
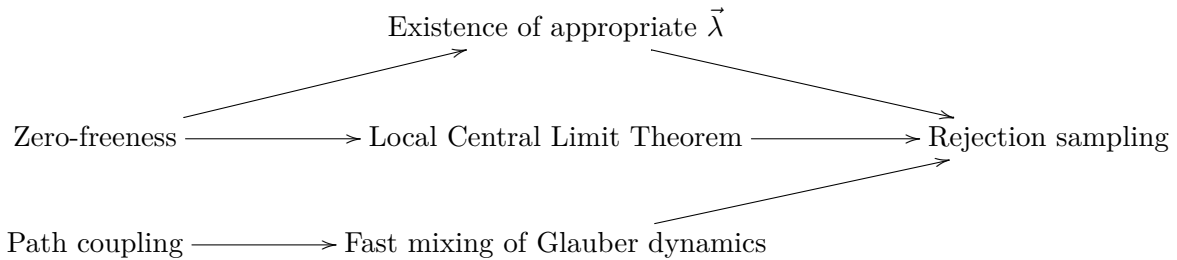

\section{Zero-freeness}\label{section: zero-freeness}
\subsection{Overview}
Recall the partition function and Gibbs measure of the coloring model:
    \[Z_G(\lambda_1, \dots, \lambda_q) \defeq \sum_{\substack{\sigma \in [q]^V \\ \sigma \text{ proper}}}  \prod_{i = 1}^q \lambda_i^{|\sigma^{-1}(i)|}, ~~~~ \P_{G, \vl}(\sigma) \defeq \frac{\prod_{i = 1}^q \lambda_i^{|\sigma^{-1}(i)|}}{Z_G(\lambda_1, \dots, \lambda_q)} \]

For it to have a probabilistic interpretation, we would need $\lambda_i \in \mathbb R^+$, but we can still consider the  partition function as a multivariate polynomial over the complex numbers.

\begin{remark}
    At $\vec \lambda= \vec{1} := (1, \dots, 1)$, the partition function is counting all proper $q$-colorings.
    Since $q \geq 2\Delta$, a proper $q$-coloring always exists, and hence $Z_G(\vec{1}) \not = 0$. 
\end{remark}

The aim of this section is to show that in $\C^q$ there exists a ball $B$ around $\vec{1}$ such that $\forall \vec{\lambda} \in B$, $Z_G(\vec{\lambda}) \not = 0$. 
 
The proof here uses the framework of Liu, Sinclair and Srivastava~\cite{liu2025correlation} for the Potts model, in which a configuration's weight is defined by the number of monochromatic edges of the coloring. 
 
A variable $\omega$ defines the ``penalty'' for the monochromatic edge. Let $m_G(\sigma)$ be the number of monochromatic edges introduced by the coloring $\sigma$. The partition function of the Potts model is a sum over all colorings (including improper), where the weight of a coloring is proportional to $\omega^{m_G(\sigma)}$. In other words, the partition function for their Potts model is
\[Z_G(\omega) \defeq \sum_{\sigma \in [q]^V} \omega^{m_G(\sigma)} \]
Note that the partition function is univariate. In our case the partition function is multivariate since every color has its own weight, thus a part of our work is to lift several key ingredients of the proof to the multivariate case. Note however, that our model does not allow for the monochromatic edges, so the penalty for such edges is simply $\omega = 0$.

Since all the steps had to be adapted for our model, we will provide detailed calculations, and we utilize the structure and definitions of \cite{liu2025correlation} for the acquainted reader to track differences easily. 

\subsection{Main result}
Throughout this section we fix the constants $\nu = 0.9$, $\eps_R \defeq 10^{-2} \Delta^{-2}$, and $\eps_I \defeq 10^{-4} \Delta^{-4}$. We will use $\ln$ to denote the principal branch of the logarithm, which is analytic on $\mathbb C \setminus (-\infty, 0]$ and satisfies $\arg(z) := \Im \ln z \in (-\pi, \pi)$.

We will commonly refer to the following assumptions on $\vec \lambda \in \mathbb C^q$.
\begin{definition}[Assumptions on $\vec \lambda$]\label{assumptions} We say that $\vec{\lambda} \in \mathbb C^q$ is \textit{valid} if, for any $i, j \in [q]$,
\begin{enumerate}
    \item $\lambda_i \notin (-\infty, 0] \subset \mathbb R $ \label{assumption:lambdas-nonzero}
    \item $\left\vert \arg(\lambda_i) \right\vert \leq \nu \eps_I / 2$ \label{assumption:lambda-argument-small}
    \item $|\Re \ln \frac{\lambda_i}{\lambda_j}| \leq \Delta \eps_I$ \label{assumption:lambdas-close-together}
\end{enumerate}
We will denote $\lambda_\text{min} := \min\{1, \, \min_{i \in [q]} \left \vert \lambda_i \right \vert\}$ and $\lambda_{\text{max}} := \max\{1, \, \max_{i \in [q]} \left \vert \lambda_i \right \vert\}$.
\end{definition}

\begin{remark}
    Note that we need assumption~\ref{assumption:lambdas-nonzero} for the quantity in assumption~\ref{assumption:lambda-argument-small} to be well-defined, and then $\left\vert \arg(\lambda_i/\lambda_j) \right\vert \leq \nu \varepsilon_I < \pi$, so assumption~\ref{assumption:lambdas-close-together} is also well-defined. 
\end{remark}

The main result we will prove is the following:
\begin{theorem}[(Zero-freeness around $\vec{1}$)]\label{theorem:zero-freeness}
    Let $\Delta \geq 1$, $q \geq 2\Delta$, and let $G$ be a graph with maximum degree $\Delta$. Assume that $\vl \in \C^{q}$ is valid, in the sense of \autoref{assumptions}. Then, $Z_G(\vl) \neq 0$.
\end{theorem}

The assumptions from \autoref{assumptions} are expressed in terms of the logarithm of the ratio between components of $\vec \lambda$. That may not be practical for some applications, but it can be shown that any $\vec \lambda \in \mathbb C^q$ sufficiently close to $\vec 1$ (for example, in $\ell_\infty$ distance) satisfies those.
\begin{lemma}\label{lemma:lambdas-in-ball-satisfy-assumptions}
    Let $R := \nu \varepsilon_I / 4$. Then, any $\vec \lambda \in \mathbb C^q$ such that $\vert \lambda_i - 1 \vert \leq R$ for every $i \in [q]$ is valid, in the sense of \autoref{assumptions}.
\end{lemma}
Hence, we may derive the following corollary from \autoref{theorem:zero-freeness}: 
\begin{corollary}[(Zero-freeness radius)]\label{corollary:zero-freeness-radius}
      Let $\Delta \geq 1$, $q \geq 2\Delta$, and let $G$ be a graph with maximum degree $\Delta$. Then, for any $\vec \lambda \in B_{\ell_{\infty}}(\vec 1, R)$, where $R = \nu \varepsilon_I / 4$, we have that $Z_G(\vl) \neq 0$. 
\end{corollary}

\begin{proof}[Proof of \autoref{lemma:lambdas-in-ball-satisfy-assumptions}]
    Let $i \in [q]$. Note that $\vert \lambda_i - 1 \vert \leq R$ implies that $\left\vert \Re\{\lambda_i - 1 \} \right \vert \leq R$, so $\Re \lambda_i \geq 1 - R > 0$.

    For assumption~\ref{assumption:lambdas-close-together}, notice that it is enough to show that $ \vert \Re \ln \lambda_i \vert \leq \Delta \varepsilon_I / 2$ for all $i \in [q]$. From the hypothesis on $\vec \lambda$, we have that $1 - R \leq \vert \lambda_i \vert \leq 1 + R$, so
    \begin{equation*}
        \Re \ln \lambda_i = \ln \vert \lambda_i \vert \leq \ln(1 + R) \leq R \leq \frac{\nu \varepsilon_I}{4} < \frac{\Delta \varepsilon_I}{2}
    \end{equation*}
    On the other hand,
    \begin{equation*}
        \Re \ln \lambda_i \geq \ln(1 - R) \geq -2R > -\frac{\Delta \varepsilon_I}{2}
    \end{equation*}
    where we have used that $\ln(1 - x) \geq -2x$ for $0 \leq x \leq 1/2$, which can be easily seen by comparing derivatives.

    For assumption~\ref{assumption:lambda-argument-small}, since the region $\{z \in \mathbb C \, : \, \vert z - 1 \vert \leq c \}$ is symmetric with respect to the real axis, we may assume without loss of generality that $\arg(\lambda_i) \in [0, \pi]$. Then,
    \begin{align*}
        \tan \left( \arg(\lambda_i) \right) \leq \frac{R}{1-R} \leq 2R
    \end{align*}
    where we used that $R < 1/2$. Finally, since $\arctan x$ is increasing on $x$ and $x \geq \arctan x \geq 0$ for all $x \geq 0$, we conclude that
    \begin{equation*}
       \arg(\lambda_i) \leq \arctan 2R \leq 2R \leq \frac{\nu \varepsilon_I}{2}. 
    \end{equation*}
\end{proof}
\begin{remark}
    The constants used in the proof of \autoref{theorem:zero-freeness} are not particularly optimized, so the zero-freeness radius is susceptible to improvement.
\end{remark}

\begin{remark}\label{remark:existence-branch-for-logZ}
    According to \autoref{corollary:zero-freeness-radius}, $Z_G(\vec \lambda)$ is zero-free within the simply-connected region $B = B_{\ell_\infty}(\vec 1, R) \subseteq \mathbb C^q$. Hence, there exists a branch of the logarithm such that $\log Z_G(\vec \lambda)$ is well-defined and analytic for $\vec \lambda \in B$ (i.e. an analytic $f:B \longrightarrow \mathbb C$ satisfying $e^{f(\vec \lambda)} = Z_G(\vec \lambda)$). In the proofs in this section we are always able to take the principal branch of the logarithm, but that will not be the case for the results in the next section.  
\end{remark}

The inductive proof not only shows that $Z_G(\vec \lambda) \neq 0$, but actually gives a lower bound on its norm (which, despite being exponentially decreasing in $n$, will be of use later in bounding $\big \vert \log Z_G(\vec \lambda) \big\vert $).

\begin{corollary}\label{corollary:lower-bound-norm-partition-function}
    Let $\Delta \geq 1$, $q \geq 2 \Delta$, and let $G$ be a graph with $n$ vertices and maximum degree $\Delta$. Let $\vec \lambda \in \mathbb C^q$ be valid, according to \autoref{assumptions}. Then,
    $$
        \left \vert Z_G(\vec \lambda) \right \vert \geq 0.99^n\left( \frac{\lambda_{\text{min}}}{\lambda_{\text{max}}} \right)^{n\Delta},
    $$
    where $\lambda_\text{min} := \min\{1,\, \min_{i \in [q]} \left \vert \lambda_i \right \vert\}$ and $\lambda_\text{max} := \max\{1,\, \max_{i \in [q]}\left\vert \lambda_i \right\vert\}$.
\end{corollary}

\subsection{Partial colorings and recursive structure}
We will prove \autoref{theorem:zero-freeness} with a slightly stronger assumption; some vertices of the graph may already be ``pinned'' to a color. That will allow us to proceed by induction on the number of unpinned vertices of the graph. We will formalize this idea through the following definition:
\begin{definition}[Partially-$q$-colored graph]
    Let $q \geq 1$. A \textit{partially-$q$-colored graph} $(G, \tau)$ is a graph $G$ together with a partial $q$-coloring $\tau : V(G) \longrightarrow[q] \cup \{\ast\}$. The partial coloring needs to be \textit{proper}, that is, for every edge $uv \in E(G)$ either $\tau(u) \neq \tau(v)$ or $\tau(u) = \tau(v) = \ast$ need to hold. Vertices $v \in V(G)$ such that $\tau(v) = \ast$ are said to be \textit{unpinned}. Vertices such that $\tau(v) = c \in [q]$ are said to be \textit{pinned to color $c$}. We additionally require that $\deg_G(v) = 1$ for all pinned $v$.
\end{definition}

\begin{remark}
    We may extend our model to a partially-colored graph $(G, \tau)$ by restricting the partition function to sum over the colorings that agree with $\tau$ on the pinned vertices. That is, we define
    \begin{equation*}
        Z_{G, \tau}(\vec \lambda) \defeq \sum_{\substack{\sigma \in [q]^{V(G)} \\ \sigma(v) = \tau(v) \text{ for pinned $v$}}} \prod_{v \in V(G)} \lambda_{\sigma(v)}
    \end{equation*}
    We will use $\mathcal C_{G, \tau}$ to denote the set of proper $q$-colorings of $G$ which agree with $\tau$ on the pinned vertices.
\end{remark}

We will also introduce the following operation, which from a partially-$q$-colored graph $(G, \tau)$ generates another partially-$q$-colored graph $(\tilde G, \tilde \tau)$ which has one unpinned vertex fewer:
\begin{definition}[Pinning operation]
    Let $v$ be a vertex in $(G, \tau)$ with degree $d$. To \textit{pin vertex $v$ to color $c$} means that we have constructed a new graph $\tilde G$ in which we replaced the vertex $v$ with $d$ copies $v_1, \dots, v_{d}$ (each of them joined to a different neighbor of $v$), and which is equipped with a partial coloring $\tilde \tau$ defined so that $\tilde \tau(v_i) = c$ for all $i \in [d]$ and $\tilde \tau(w) = \tau(w)$ for all other $w \in V(\tilde G) \setminus \{v_1, \dots, v_d\} = V(G) \setminus \{v\}$.
\end{definition}
\begin{remark}
    Note that, after applying the previous operation, the newly-pinned vertices have all degree 1, as required in the definition of partially-colored graph.
\end{remark}

Since we are working with a model that only allows proper colorings, we will never want to pin a vertex to a color which is already present in its neighborhood. Thus, following the notation from \cite{liu2025correlation}, we will classify the colors as either \textit{good} or \textit{bad} (for a given vertex):
\begin{definition}[Good color, bad color]
    A $c \in [q]$ is a \textit{good color for vertex $v$} if it is different from all colors to which the neighbors of $v$ are pinned to. Otherwise, it is called a \textit{bad color}. The set of good colors of a vertex $v$ will be denoted as $\Gamma_v$, while the set of bad colors will be denoted as $B_v$. 
\end{definition}

We will denote $(G, \tau)$ as $G$, despite the abuse of notation, whenever it is clear from context that we are referring to a partially-colored graph and there is no ambiguity with respect to the underlying partial coloring $\tau$. That extends as well to $Z_G$ (in place of $Z_{G, \tau}$) and $\mathcal C_G$ (in place of $\mathcal C_{G, \tau}$). 

Next, we develop the definitions needed to frame the zero-freeness problem recursively, in terms of graphs which have fewer unpinned vertices.

\begin{definition}[Restricted partition function, marginal pseudo-probability, marginal ratio]
    Let $G$ be a partially-$q$-colored graph. Let $\vec \lambda$ be a vector of formal variables. For a given unpinned vertex $v \in V(G)$ and a color $i \in [q]$, we define
    $$
        Z_{G,v}^{(i)}(\vl) := \sum_{\substack{\sigma \in \mathcal C_G \\ \sigma(v) = i}} \prod_{v \in V(G)} \lambda_{\sigma(v)}
    $$ 
    which corresponds to the restriction of the partition function to the colorings in which $v$ has color $i$.

    Let $\Delta \geq 1$ be the maximum degree of the graph $G$, and assume that $q \geq \Delta + 1$. Then, we define the \textit{marginal pseudo-probability}
    $$
        \mathcal P_{G, \vl}[\sigma(v) = i] := \frac{Z_{G,v}^{(i)}(\vl)}{Z_{G}(\vl)}
    $$
    and, given another color $j \in \Gamma_v$, we define the \textit{marginal ratio}
    $$
        R^{(i,j)}_{G,v} := \frac{Z_{G,v}^{(i)}(\vl)}{Z_{G,v}^{(j)}(\vl)}.
    $$
\end{definition} 
\begin{remark}
    All the expressions above are to be interpreted as quotients of polynomials of the vector of formal variables $\vec \lambda$. Note that, under this interpretation, $Z_{G, v}^{(i)}(\vec \lambda) = 0$ for any $i \in B_v$. On the other hand, if $i \in \Gamma_v$ and $q \geq \Delta + 1$, we know that there is at least one proper coloring of $G$ which is consistent with the current partial coloring, so $Z_{G, v}^{(i)}(\vec \lambda) \neq 0$ (and, by the same logic, $Z_G(\vec \lambda) \neq 0$). Hence, both the marginal pseudo-probability and the marginal ratios are well-defined. Nonetheless, it could still happen that $Z_{G, v}^{(i)}(\vec \lambda) = 0$ and/or $Z_G(\vec \lambda) = 0$ once we substitute the formal variables by a specific value $\vec \lambda \in \mathbb C^q$.
\end{remark}

\begin{remark}
    The name and notation for the marginal pseudo-probability comes from the fact that, for real positive $\vl$, this is the marginal probability that $v$ has color $i$ under our model. Note, however, that for a general $\vl \in \mathbb C^q$, this quantity has no direct probabilistic meaning.
\end{remark}

We would like to relate $R^{(i,j)}_{G, v}(\vec \lambda)$ to the partition function of a graph with fewer unpinned vertices than $G$. The numerator $Z_{G,v}^{(i)}(\vec \lambda)$ can be related to the graph in which we pin vertex $v$ to color $i$, while the denominator $Z^{(j)}_{G, v}(\vec \lambda)$ can be related to the graph in which we pin vertex $v$ to color $j$. Next, we define a sequence of graphs that allows us to interpolate between these two cases.

\begin{definition}[Graph $G_k^{(i,j)}$]\label{definition:graph_for_telescoping_recurrence}
    Given a partially-colored graph $G$, an unpinned vertex $v$ with degree $d$, and an ordering of the neighbors of $v$ (denoted as $w_1, \dots, w_d$), for every $k \in [d]$ and for every $i,j \in \Gamma_v$ we define the partially-colored graph $G_k^{(i,j)}$ obtained from $G$ by replacing $v$ with $d$ copies $v_1, \dots, v_{d}$, attaching  each $v_\ell$ to $w_\ell$, pinning vertices $v_1, \dots, v_{k-1}$ to color $i$, pinning vertices $v_{k+1}, \dots, v_{d}$ to color $j$, and deleting $v_k$.  
\end{definition}

One can then derive a recurrence relation that expresses the ratios in $G$ in terms of the marginal pseudo-probabilities in each of the $G_k^{(i,j)}$.

\begin{lemma}[Recurrence relation] \label{lemma: recurrence} 
Let the vertices $w_1, \dots, w_{\deg_G(v)}$ be the neighbors of vertex $v$ in the graph $G$. Then,
    \[R^{(i,j)}_{G,v}(\vec{\lambda}) = \frac{\lambda_i}{\lambda_j}\prod_{k=1}^{\deg_G(v)}\frac{ 1 - \mathcal P_{G_k^{(i,j)},\vl}[\sigma(w_k) = i ] }{1 - \mathcal P_{G_k^{(i,j)},\vl}[\sigma(w_k) = j ]  }
  \]
\end{lemma}

\begin{remark}
The recurrence relation holds under the interpretation that $\vec \lambda$ is a vector of formal variables. It will also hold after plugging in a particular $\vec \lambda \in \mathbb C^q$, as long as all the terms are well-defined, that is,
\begin{itemize}
    \item $Z_{G, v}^{(j)}(\vec \lambda) \neq 0$,
    \item $Z_{G_k^{(i,j)}}(\vec \lambda) \neq 0$, $\quad$ and
    \item $\mathcal P_{G_k^{(i,j)}, \vec \lambda}[\sigma(w_k) = j] \neq 1$ for all $k \in [\deg_G(v)]$.
\end{itemize}
\end{remark}
\begin{proof}
    Let us first define an auxiliary graph $H_k^{(i,j)}$, which is analogous to $G_k^{(i,j)}$, but in which $v_k$ is pinned to color $i$, instead of being deleted. 
    \begin{definition}[Graph $H_k^{(i,j)}$]\label{definition:auxiliary_graph_for_telescoping_recurrence}
        Given a partially-colored graph $G$, an unpinned vertex $v$ with degree $d$, and an ordering $w_1, \dots, w_d$ of the neighbors of $v$, we define the graph $H_k^{(i,j)}$ by replacing vertex $v$ with $d$ copies $v_1, \dots, v_d$, attaching each $v_\ell$ to $w_\ell$, pinning vertices $v_1, \dots, v_k$ to color $i$ and pinning vertices $v_{k+1}, \dots, v_d$ to color $j$. 
    \end{definition}
     
    Note that $H_0^{(i,j)}$ and $H_d^{(i,j)}$ correspond to the edge cases in which all copies of $v$ are colored $j$, or all copies are colored $i$, respectively. Hence, $Z_{H_d^{(i, j)}}(\vl) = \lambda_i^{d-1} Z_{G,u}^{(i)}(\vl)$ and $Z_{H_0^{(i, j)}}(\vl) = \lambda_j^{d-1} Z_{G,u}^{(j)}(\vl)$. Therefore, we may write the marginal ratio as
    \[ R^{(i,j)}_{G,u}(\vl) = \frac{Z_{G,u}^{(i)}(\vl)}{Z_{G,u}^{(j)}(\vl)} = \frac{\lambda_j^{d-1}}{\lambda_i^{d-1}} \frac{Z_{H_{d}^{(i, j)}}(\vl)}{Z_{H_0^{(i, j)}}(\vl)} = \frac{\lambda_j^{d-1}}{\lambda_i^{d-1}} \prod_{k = 1}^{d} \frac{Z_{H_k^{(i,j)}}(\vl)}{Z_{H_{k-1}^{(i,j)}}(\vl)}  \]
    There is a correspondence between proper colorings of $H_k^{(i,j)}$ and proper colorings of $G_k^{(i,j)}$ in which $\sigma(w_k) \neq i$. Therefore, \[Z_{H_k^{(i,j)}}(\vl) = \lambda_i(Z_{G_k^{(i,j)}}(\vl) -  Z_{G_k^{(i,j)}, w_k}^{(i)}(\vl) )\]
    and, similarly,
    \[Z_{H_{k-1}^{(i,j)}}(\vl) = \lambda_j(Z_{G_k^{(i,j)}}(\vl) -  Z_{G_k^{(i,j)}, w_k}^{(j)}(\vl) ). \] 
    Hence, 
    \[ R^{(i,j)}_{G,u}(\vl) =  \frac{\lambda_j^{d-1}}{\lambda_i^{d-1}} \frac{\lambda_i^{d}}{\lambda_j^{d}} \prod_{k=1}^{d} \frac{Z_{G_k^{(i,j)}}(\vl) -  Z_{G_k^{(i,j)}, w_k}^{(i)}(\vl) }{Z_{G_k^{(i,j)}}(\vl) -  Z_{G_k^{(i,j)}, w_k}^{(j)}(\vl)} = 
    \frac{\lambda_i}{\lambda_j} \prod_{k=1}^{d} \frac{1 - P_{G_k^{(i,j)},\vl}[\sigma(w_k) = i ]}{1 - P_{G_k^{(i,j)},\vl}[\sigma(w_k) = j ]}.
    \]
\end{proof}

\subsection{Complex analysis tools}

In this subsection we introduce the ingredients from complex analysis we will need for the proof. The first three lemmas are from \cite{liu2025correlation}, while the last three are refinements of simple facts that were implicitly used in \cite{liu2025correlation}. 
\\

Let $D$ be a domain in $\mathbb C$ with the following properties:
\begin{itemize}
  \item For any $z \in D$, $\Re(z) \in D$.
  \item For any $z_1, z_2\in D$, there exists a point $z_0 \in D$ such that one of the numbers $z_1 - z_0, z_2 - z_0$ has zero real part while the other has zero imaginary part.
  \item If $z_1, z_2 \in D$ are such that either $\Im(z_1) = \Im(z_2)$ or $\Re(z_1) = \Re(z_2)$, then the segment $[z_1, z_2]$ lies in $D$.
\end{itemize}
As remarked by Liu, Sinclair and Srivastava \cite{liu2025correlation}, a rectangular region symmetric about the real axis will satisfy all of the above properties.

\begin{lemma}[(Mean value theorem for complex functions - Lemma 3.5 from \cite{liu2025correlation})]
Let $f$ be a holomorphic function on a domain~$D$ as above such that, for $z\in D$, $\Im(f(z))$ has the same sign as $\Im(z)$.  Suppose further that there exist positive constants $\rho_I$ and $\rho_R$ such that
  \begin{itemize}
  \item for all $z \in D$,  $|\Im(f'(z))| \le \rho_I$;
  \item for all $z \in D$, $\Re(f'(z)) \in[0, \rho_R]$. 
\end{itemize}
  Then for any $z_1, z_2 \in D$, there exists $C_{z_1, z_2} \in [0, \rho_R]$ such that
  \begin{align*}
	  |\Re(f(z_1) - f(z_2)) - C_{z_1, z_2} \cdot \Re(z_1 - z_2)| \leq  \rho_I \cdot |\Im(z_1 - z_2)|,
  \end{align*}
  and furthermore,
  \begin{align*}
    |\Im(f(z_1) - f(z_2))| \leq \rho_R \cdot
    \begin{cases}
      |\Im (z_1 - z_2)|, \quad &\text{ when $\Im(z_1)\cdot \Im(z_2) \leq 0$;}\\
      \max\{|\Im(z_1)|, |\Im(z_2)|\}, \quad & \text{ otherwise.}
    \end{cases}
\end{align*}
  \label{lemma: mean}
\end{lemma}

The first point of the next lemma is needed to show that the function 
\begin{equation} \label{eq:definition_of_f}
    f(x) \defeq -\ln(1 - e^x)
\end{equation} satisfies the conditions of the previous lemma.

\begin{lemma}[(Lemma 3.6 from \cite{liu2025correlation})]\label{lemma: f}
  Consider the domain $D$ given by
  \begin{displaymath}
    D = \{z ~|~ \Re(z) \in (-\infty, -\zeta)\text{ and } |\Im(z)| < \tau\},
  \end{displaymath}
  where $\tau < 1/2$ and $\zeta$ are positive real numbers such that
  $\tau^2 + e^{-\zeta} < 1$, and the
  function $f$ defined in eq. \eqref{eq:definition_of_f}. These satisfy the hypotheses of \autoref {lemma: mean}, if $\rho_R$ and $\rho_I$ in the statement of the
    theorem are taken to be $\frac{e^{-\zeta}}{1 - e^{-\zeta}}$ and
    $\frac{\tau e^{-\zeta}}{\left(1 - e^{-\zeta} \right)^2}$, respectively.
\end{lemma}

The next lemma gives us a way to lower-bound the sum of complex numbers, if we know that the angles between them are small enough. 

\begin{lemma}[(Lemma 3.7 from \cite{liu2025correlation})]\label{lemma:angles} Let $z_1, z_2, \dots, z_n$ be complex numbers such that the angle between any
  two non-zero $z_i$ is at most $\alpha \in [0, \pi/2)$.  Then
  $ \abs{\sum_{i=1}^n z_i} \geq \cos(\alpha/2)\sum_{i=1}^n \abs{z_i}$.
  \label{lem:geometric-proj}
\end{lemma}

An important part of the proof will require going from bounds on $z$ to bounds on $e^z$, and viceversa. That's what the following two lemmas provide.
\begin{lemma} \label{lemma:real_and_imag_bounds_to_exponential}
     Let $r, \theta \in \mathbb{R}$ such that $\left\vert \theta \right\vert \leq \theta_0 < \pi$ and $\vert r \vert \leq r_0$ for certain $r_0, \theta_0 \in \mathbb R^{\geq 0}$. Let $z:= r + i \theta$. Then, $ \left\vert \arg(e^z) \right\vert \leq \theta_0$ and $e^{-r_0} - \theta_0^2/2 \leq \Re(e^z) \leq  e^{r_0}$. Furthermore, if $\theta_0 \leq \ln 2 \cong 0.693$, then we also have that $e^{-r_0 -\theta_0^2} \leq \Re(e^z)$.
\end{lemma}
\begin{proof}
    Clearly, $\arg(e^z) = \arg(e^{i \theta}) = \theta$. Thus, $\left\vert \arg(e^z) \right\vert \leq \theta_0$. For the real part, note that $\Re(e^z) = e^r \cos \theta \leq e^{r_0}$. Besides, $\cos x \geq 1 - x^2 /2$ for all $x \in \mathbb R$ (as can be seen by comparing the first and second derivatives), so $\Re(e^z) \geq e^{r}(1 - \theta^2/2) \geq e^{-r_0}(1 - \theta^2/2) \geq e^{-r_0} - \theta_0^2/2$.

    When $0 \leq \theta_0 \leq \ln 2$, we have that $1 - \theta^2/2 \geq e^{-\theta^2}$ (again, by a trivial comparison of the first derivatives), so we get the additional lower-bound of $\Re(e^z) \geq e^{-r_0}(1 - \theta_0^2/2) \geq e^{-r_0 - \theta_0^2}$.
\end{proof}

\begin{lemma}\label{lemma:reintroduce_logs}
    Let $z \in \mathbb{C}$ with $\Re(z) \neq 0$ and $\left \vert \arg(z) \right\vert \leq \theta \leq 0.1$. Then, $\ln(\Re(z)) \leq \Re(\ln(z)) \leq \ln(\Re(z)) + \theta^2$ and $\left \vert \Im(\ln(z)) \right\vert \leq \theta$.
\end{lemma}
\begin{proof}
    The second statement is trivial, since $\Im(\ln(z)) = \arg(z)$. For the first, note that $\Re(\ln(z)) = \ln(\vert z \vert)$, so
    \begin{align*}
        \ln(\Re(z)) &= \ln\big(\vert z \vert \cos(\arg(z))\big) = \Re(\ln(z)) + \ln\big(\cos(\arg(z))\big) \leq \Re(\ln(z)) , \quad \text{and} \\
        \ln(\Re(z)) &= \Re(\ln(z)) + \ln\big(\cos(\arg(z))\big) \geq \Re(\ln(z)) + \ln(\cos(\theta)) \geq \Re(\ln(z)) + \ln\Big(1 - \frac{\theta^2}{2}\Big) \\
        &\geq \Re(\ln(z)) - \frac{\theta^2/2}{1 - \theta^2/2} \geq \Re(\ln(z)) - \theta^2
    \end{align*}
    where we have used that $\theta < 1$, that $\ln(1 + x) \geq x/(1 + x)$ for $x > -1$, and that $\cos(x) \geq 1 - x^2/2$ for all $x \in \mathbb R$.
\end{proof}

Finally, a result in the same vein as \autoref{lemma:angles}.
\begin{lemma}\label{lemma:argument_of_sum}
    Let $z_1, \dots, z_k \neq 0$ be complex numbers such that $\left \vert\arg(z_i)\right\vert \leq \theta$ for all $i \in [k]$ and some $\theta < \pi/2$. Then, $\sum_{i = 1}^k z_i \neq 0$ and $\left \vert \arg \Big( \sum_{i = 1}^k z_i \Big) \right \vert \leq \theta $.
\end{lemma}
\begin{proof}
The condition on the argument means that $\Re z_i > 0$ for all $i \in [k]$. Let $m := \max\limits_{i \in [k]} \left\vert \Im z_i \right\vert / \Re z_i$. Then,
\begin{align*}
    \left \vert \arg\Big(\sum_{i = 1}^k z_i \Big) \right \vert &= \left \vert \arctan \left( \frac{\sum_{i \in [k]} \Im z_i}{\sum_{i \in [k]} \Re z_i} \right) \right\vert \leq \arctan \left( \frac{\sum_{i \in [k]} \left \vert \Im z_i \right\vert}{\sum_{i \in [k]} \Re z_i}\right) \leq \arctan \left( \frac{\sum_{i \in [k]} m \Re z_i}{\sum_{i \in [k]} \Re z_i}\right) \\
    &= \arctan m \leq \theta
\end{align*}
where we have used that $\arctan(-x) = \arctan(x)$ and that it is an increasing function for $x \geq 0$.
\end{proof}

\subsection{Induction hypothesis}

The main idea of the proof of the zero-freeness result of \autoref{theorem:zero-freeness} is to pick a vertex $u$ and rewrite the partition function as a sum of $Z_{G,u}^{(i)}(\vl)$ for all $i \in \Gamma_u$. If we can prove that the angles between these terms are small, then it is enough for one of the terms to be non-zero for the whole sum to be non-zero (see \autoref{lemma:angles}). 

Both the zero-freeness of $Z_{G,u}^{(i)}(\vec \lambda)$ and the fact that the angles are small are shown by induction on the number of unpinned vertices of the graph. For the first, we simply use that $Z_{G,u}^{(i)}(\vec \lambda) = \lambda_i^{\deg(u)-1} Z_{\tilde G}(\vec \lambda)$, where $\tilde G$ is the graph obtained by pinning vertex $u$ to color $i$. For the second, we need to bound the argument of the ratios $R^{(i,j)}_{G, u}$. For that purpose, we use \autoref{lemma: recurrence}, in which we derived a recursive formula that expresses $R^{(i,j)}_{G, u}$ in terms of marginal pseudo-probabilities of certain graphs $G_k^{(i,j)}$, which all have one fewer unpinned vertex than $G$. 

The following lemma summarizes the statements constituting our induction hypothesis.
\begin{lemma}\label{lemma:induction}
    Let $G$ be a partially-$q$-colored graph with maximum degree $\Delta \geq 1$ and $q \geq 2 \Delta$. Let $u$ be an unpinned vertex of $G$. Then, for a valid $\vec \lambda \in \mathbb C^q$, in the sense of \autoref{assumptions},
    \begin{enumerate}
        \item \label{induction-item:zero-freeness} For all $i \in \Gamma_u$, $$\left|Z_{G,u}^{(i)}(\vl)\right| \geq 0.99^{\ell-1} \frac{\lambda_\text{min}^{n + \ell (\Delta-1)}}{\lambda_\text{max}^{\ell\Delta}} > 0,$$
        where $\ell$ the number of unpinned vertices and $n$ the total number of vertices of $G$. 
        \item \label{induction-item:ratio-all-pinned} For $i, j \in \Gamma_u$, if $u$ has all of its neighbors pinned, then $R^{(i,j)}_{G,u}(\vl) = \lambda_i / \lambda_j$.
        \item \label{induction-item:log-ratio-well-defined} For $i, j \in \Gamma_u$, $R_{G, u}^{(i,j)}(\vec \lambda) \in \mathbb C \setminus (-\infty, 0]$.
        \item \label{induction-item:real-part} For $i, j \in \Gamma_u$, if $u$ has $d_u $ unpinned neighbors, then 
        \[ \left|\Re(\ln R^{(i,j)}_{G,u}(\vl)) - \ln R_{G, u}^{(i,j)}(\vec 1)\right| \leq d_u \eps_R + \left|\Re \ln \frac{\lambda_i}{\lambda_j}\right|.    \]
        \item  \label{induction-item:imag-part} For $i, j \in \Gamma_u$, if $u$ has $d_u$ unpinned neighbors, then $$\left|\Im(\ln R^{(i,j)}_{G,u}(\vl) )\right| \leq d_u \eps_I + \left| \Im \ln  \frac{\lambda_i}{\lambda_j}\right|.$$
        \item \label{induction-item:ratio-for-bad-colors}  For $i \not \in \Gamma_u, j \in \Gamma_u$, we have that $R^{(i,j)}_{G,u}(\vl) = 0$.
    \end{enumerate}   
\end{lemma}

\begin{remark}
    The logarithm of $\lambda_i / \lambda_j$ is well-defined due to the assumptions on $\vec \lambda$, while the ratios $R_{G, u}^{(i,j)}(\vec \lambda)$ are well-defined due to item~\ref{induction-item:zero-freeness}. We can take the logarithm of $R_{G,u}^{(i,j)}(\vec \lambda)$ due to item~\ref{induction-item:log-ratio-well-defined}.
\end{remark}
\begin{remark}
    Note that $Z_{G, u}^{(i)}(\vec 1)$ counts the number of colorings of $G$ which have $\sigma(u) = i$. For $i \in \Gamma_u$ and $q \geq \Delta + 1$, that number is non-zero, so $R_{G, u}^{(i,j)}(\vec 1)$ is a quotient of positive reals, and hence $\ln R_{G, u}^{(i,j)} \in \mathbb R$.
\end{remark}

The zero-freeness of the full partition function follows easily once \autoref{lemma:induction} is proved:

\begin{corollary}\label{corollary:zero-freeness-partially-colored-graph}
    Let $G$ be a partially-$q$-colored graph with maximum degree $\Delta \geq 1$ and $q \geq 2 \Delta$. Let $u$ be an unpinned vertex of $G$. Then, \[|Z_{G}(\vl)| \geq 0.99 \min_{i \in \Gamma_u} \left\{|Z_{G,u}^{(i)}(\vl)|  \right\} > 0.99^\ell \frac{\lambda_\text{min}^{n + \ell(\Delta -  1)}}{\lambda_{\text{max}}^{\ell \Delta}} > 0, \]
    where $\ell$ is the number of unpinned vertices and $n$ the total number of vertices of $G$.
\end{corollary}

\begin{proof}
    Observe that  $Z_{G,u}^{(i)}(\vec \lambda) = 0$ for $i \in B_u$. Hence, $Z_{G}(\vl) = \sum_{i \in \Gamma_u} Z_{G,u}^{(i)}(\vl)$. Since $\Im(\ln z) = \arg(z)$, from item~\ref{induction-item:imag-part} of \autoref{lemma:induction} we know that the angle between the terms is at most 
    \begin{align*}
        \left\vert \arg \left( R_{G,u}^{(i, j)}(\vec \lambda) \right) \right\vert \leq \Delta \varepsilon_I + \left \vert \Im \ln \frac{\lambda_i}{\lambda_j} \right \vert \leq (\Delta + \nu) \varepsilon_I
    \end{align*}
    where we also used assumption~\ref{assumption:lambda-argument-small} from \autoref{assumptions}.
    Using \autoref{lemma:angles} together with the fact that $(\Delta + \nu) \varepsilon_I/2 \leq \Delta \varepsilon_I \leq 10^{-4} \leq \arccos 0.99$, we have that
    $$\left|\sum_{i \in \Gamma_u} Z_{G,u}^{(i)}(\vl)\right| \geq \cos\big((\Delta + \nu) \eps_I /2 \big) \sum_{i \in \Gamma_u} \left\vert Z_{G,u}^{(i)}(\vl)\right\vert \geq 0.99 \min_{i \in \Gamma_u} \left|Z_{G,u}^{(i)}(\vl)\right|$$ 
    The result then follows by item~\ref{induction-item:zero-freeness} of \autoref{lemma:induction}. 
\end{proof}
The previous corollary states that the partition function is zero-free for any partially-colored graph. Then, our main theorem follows as the particular case in which the partial coloring is ``empty'':
\begin{proof}[Proof of \autoref{theorem:zero-freeness}]
  Let $G$ be a graph of maximum degree $\Delta \geq 1$, and let $q \geq 2 \Delta$. By definition, $Z_G(\vec \lambda) = Z_{G, \tau}(\vec \lambda)$, where $(G, \tau)$ is the partially-colored graph in which $\tau(v) = \ast$ for all $v \in V(G)$. Using \autoref{corollary:zero-freeness-partially-colored-graph}, $Z_{G, \tau}(\vec \lambda) \neq 0$, from which the desired conclusion follows.
\end{proof}

\autoref{corollary:lower-bound-norm-partition-function} follows analogously by taking $\ell = n$ in \autoref{corollary:zero-freeness-partially-colored-graph}, since all vertices are unpinned. Note that \autoref{corollary:zero-freeness-partially-colored-graph} is not only used now for the final zero-freeness conclusion, but also will be used at every step of the induction. 

\subsection{Consequences of the induction hypothesis}

The remainder of the section deals with the inductive proof of \autoref{lemma:induction}. First, we will proof a series of consequences of \autoref{lemma:induction}, which will be used to carry out the induction step. We will assume throughout that $G$ is a partially-colored graph of maximum degree $\Delta \geq 1$ for which \autoref{lemma:induction} holds, that $q \geq 2 \Delta$, that $u$ is an unpinned vertex of $G$ with $d_u$ unpinned neighbors, and that $\vec \lambda \in \mathbb C^q$ is valid, in the sense of \autoref{assumptions}.

\begin{lemma}\label{lemma:quotient-of-pseudoprobabilities-well-defined}
    For any $i \in \Gamma_u$,
    \begin{itemize}
        \item $\mathcal P_{G, \vec 1}[\sigma(u) = i] \in \mathbb R \setminus \{0\} $, and
        \item $\mathcal P_{G, \vec \lambda}[\sigma(u) = i] \in \mathbb C \setminus (-\infty, 0]$.
    \end{itemize}
\end{lemma}
\begin{proof}
    Observe that both pseudo-probabilities are well-defined due to \autoref{corollary:zero-freeness-partially-colored-graph} (which we know holds for $G$, since we are assuming \autoref{lemma:induction} does). Note that $\mathcal P_{G, \vec 1}[\sigma(u) = i] \in \mathbb R$ is the proportion of colorings in which $\sigma(u) = i$, which can't vanish for $i \in \Gamma_u$ and $q \geq \Delta + 1$.    

    For the second condition, first we see that $\mathcal P_{G, \vec \lambda}[\sigma(u) = i] \neq 0$, as a consequence of item~\ref{induction-item:zero-freeness} of \autoref{lemma:induction}, and then we note that we can rewrite
     \[ \frac{1}{P_{G, \vl}[\sigma(u) = i]} =  \sum_{j \in \Gamma_u} R^{(j, i)}_{G,u}(\vl), \]
     where we used that $\sum_{j \in \Gamma_u} \mathcal P_{G, \vec \lambda}[\sigma(u) = j] = 1$.

     By item~\ref{induction-item:imag-part} of \autoref{lemma:induction} and item~\ref{assumption:lambda-argument-small} of \autoref{assumptions}, $\left\vert\arg(R_{G, u}^{(j,i)}(\vec \lambda)) \right \vert \leq (d_u + \nu) \varepsilon_I$ for all $j \in \Gamma_u$. Hence, \autoref{lemma:argument_of_sum} gives us that
     $$
        \arg \left( \frac{1}{\mathcal P_{G, \vec \lambda}[\sigma(u) = i]} \right) \leq (d_u + \nu) \varepsilon_I \leq 2 \Delta \varepsilon_I \leq 0.0002 < \pi
     $$
     so $\mathcal P_{G, \vec \lambda}[\sigma(u) = i] \notin (-\infty, 0)$.
\end{proof}

\begin{lemma}\label{lemma:bounds-on-real-part-ratios}
    For any $i,j \in \Gamma_u$,
    $$
        e^{-(d_u + 0.05) \varepsilon_R} \leq \frac{\Re R_{G, u}^{(i,j)}(\vec \lambda)}{R_{G, u}^{(i,j)}(\vec 1)} \leq e^{(d_u + 0.05) \varepsilon_R}
    $$
\end{lemma}
\begin{proof}
    Let $r_0 := d_u \varepsilon_R + \Delta \varepsilon_I$ and $\theta_0 := (d_u + \nu) \varepsilon_I$. Consider $z := \ln \left(R_{G, u}^{(i,j)}(\vec \lambda) / R_{G, u}^{(i,j)}(\vec 1) \right)$. Items~\ref{induction-item:real-part} and~\ref{induction-item:imag-part} of \autoref{lemma:induction} imply that
    $$
        \left \vert \Im(z) \right \vert = \left \vert \Im\ln R_{G, u}^{(i,j)}(\vec \lambda) \right \vert \leq d_u \varepsilon_I + \left \vert \Im \ln \frac{\lambda_i}{\lambda_j} \right\vert \leq \theta_0,
    $$
    and
    $$
        \left\vert \Re z \right\vert = \left \vert \Re \ln R_{G, u}^{(i,j)}(\vec \lambda) - \ln R_{G, u}^{(i,j)}(\vec 1) \right \vert \leq d_u \varepsilon_R + \left \vert \Re \ln \frac{\lambda_i}{\lambda_j} \right \vert \leq r_0,
    $$
    where, in both cases, we have used the assumptions on $\vec \lambda$ from \autoref{assumptions} for the last inequality.

    Note that $\theta_0 \leq (\Delta + 1) \varepsilon_I \leq 2 \Delta \varepsilon_I \leq 0.0002$, so \autoref{lemma:real_and_imag_bounds_to_exponential} gives us that
    $$
        e^{-r_0 - \theta_0^2} \leq \frac{\Re R_{G, u}^{(i,j)}(\vec \lambda)}{R_{G, u}^{(i,j)}(\vec 1)} \leq e^{r_0}.
    $$
    It only remains to show that the exponents in both sides are smaller than $(d_u + 0.05) \varepsilon_R$, which follows from
    $$
        r_0 + \theta_0^2 \leq d_u \varepsilon_R + \Delta\varepsilon_I + (d_u + \nu)^2\varepsilon_I^2 \leq d_u \varepsilon_R + 5 \Delta^2 \varepsilon_I \leq (d_u + 0.05) \varepsilon_R.
    $$
\end{proof}

\begin{lemma}\label{lemma:approximation-pseudoprobabilities}[Approximation of pseudo-probabilities by real probabilities]
    For any $i \in \Gamma_u$,
    \[ \left\vert \Im\left(\ln  \frac{ \mathcal P_{G, \vl}[\sigma(u) = i]}{ \mathcal P_{G, \vec{1}}[\sigma(u) = i]}\right) \right\vert \leq (d_u + \nu) \eps_I  \]
    and 
    \[ \left\vert\Re \left(\ln  \frac{ \mathcal P_{G, \vl}[\sigma(u) = i]}{ \mathcal P_{G, \vec{1}}[\sigma(u) = i]} \right) \right\vert \leq (d_u + 0.1) \varepsilon_R. \]
\end{lemma}

\begin{proof}
    Note that the logarithm of the quotient of pseudo-probabilities is well-defined due to \autoref{lemma:quotient-of-pseudoprobabilities-well-defined}. Using the same trick as in the proof of that lemma, for any $j \in \Gamma_u$ we may write
    \begin{equation}
        \frac{P_{G, \vec{1}}[\sigma(u) = j]}{P_{G, \vl}[\sigma(u) = j]} =  \sum_{i \in \Gamma(u)} \mathcal P_{G, \vec 1}[\sigma(u) = j] R^{(i,j)}_{G,u}(\vec \lambda) = \sum_{i \in \Gamma_u} \mathcal P_{G, \vec 1}[\sigma(u) = i] \frac{R_{G, u}^{(i,j)}(\vec \lambda)}{R_{G, u}^{(i,j)}(\vec 1)} \label{eq:coefficients}
    \end{equation}

    Let $\xi_{i,j} = R_{G, u}^{(i,j)}(\vec \lambda) / R_{G, u}^{(i,j)}(\vec 1)$. Using \autoref{lemma:bounds-on-real-part-ratios}, item~\ref{induction-item:imag-part} of \autoref{lemma:induction} and the assumptions from \autoref{assumptions}, we can bound the real part and argument of each $\xi_{i,j}$:
    \begin{gather*}
        \left \vert \arg(\xi_{i,j}) \right \vert = \left \vert \Im \ln R_{G, u}^{(i,j)}(\vec \lambda) \right \vert \leq (d_u + \nu) \varepsilon_I , \quad \text{and } \\
    e^{-(d_u + 0.05)\eps_R} \leq \Re \xi_{i,j} \leq e^{(d_u + 0.05) \eps_R}
    \end{gather*}
 
    Notice that the expression at the right-hand-side of eq. \eqref{eq:coefficients} is a convex combination of the $\xi_{i_,j}$. Therefore, the bounds on $\Re \xi_{i,j}$ translate on bounds for the real part of the left-hand-side of eq. \eqref{eq:coefficients}, and so does the bound on the argument (due to \autoref{lemma:argument_of_sum}). All in all, we obtain the following:  
    \begin{gather}
        \left \vert \arg \left( \frac{P_{G, \vec{1}}[\sigma(u) = j]}{P_{G, \vl}[\sigma(u) = j]}\right) \right\vert \leq (d_u + \nu) \varepsilon_I, \quad \text{and } \label{eq:bound_argument_quotient_probs}\\
        e^{-(d_u + 0.05) \eps_R} \leq \Re \left( \frac{P_{G, \vec{1}}[\sigma(u) = j]}{P_{G, \vl}[\sigma(u) = j]} \right) \leq e^{(d_u + 0.05) \eps_R} \label{eq:bound_real_part_quotient_probs}
    \end{gather}

    Noticing that $\Im \ln z = \arg(z)$, eq. \eqref{eq:bound_argument_quotient_probs} is already the statement we would like to prove. In order to reintroduce the logarithm in eq. \eqref{eq:bound_real_part_quotient_probs}, we use \autoref{lemma:reintroduce_logs}. That gives us the upper-bound
    \begin{gather*}
        \Re \ln \left( \frac{P_{G, \vec{1}}[\sigma(u) = j]}{P_{G, \vl}[\sigma(u) = j]} \right) \leq (d_u + 0.05)\varepsilon_R + (d_u + \nu)^2\varepsilon_I^2 \leq (d_u + 0.1) \varepsilon_R
    \end{gather*}
    together with the lower-bound
    \begin{gather*}
        \Re \ln \left( \frac{P_{G, \vec{1}}[\sigma(u) = j]}{P_{G, \vl}[\sigma(u) = j]} \right) \geq -(d_u + 0.05) \varepsilon_R 
    \end{gather*}
    Combining the two, we obtain the claimed bound on the absolute value of the real part.
\end{proof}

For the next result, we need that no color is ``too probable'' in vertex $w_k$ of $G_{k}^{(i,j)}$. We quantify that through the following definition, analogous to the one from Liu, Sinclair and Srivastava \cite{liu2025correlation}:

\begin{definition}[Nice vertex]
    Let $H$ be a partially-colored graph. We say vertex $v \in V(H)$ is \textit{nice} if, for any color $c \in \Gamma_v$, we have that $\mathcal P_{H, \vec{1}}[\sigma(v) = c] \leq 1 / (d_v + 2)$, where $d_v$ is the number of unpinned neighbors of $v$.
\end{definition}

\begin{lemma}\label{lemma:vk_is_nice}
    Let $i, j \in \Gamma_u$ and $k \in [\deg_G(u)]$. Consider the graph $G_{k}^{(i,j)}$ from \autoref{definition:auxiliary_graph_for_telescoping_recurrence}, obtained from graph $G$ and unpinned vertex $u$. Let $w_1, \dots, w_k$ denote the neighbors of $u$ in $G$. Then, the vertex $w_k$ is nice in  $G_{k}^{(i,j)}$.
\end{lemma}

\begin{proof}
    Let $d := \deg_{G_k^{(i,j)}}(w_k)$ and let $\tilde d$ be its number of unpinned neighbors. By definition of the graph, $w_k$ has one neighbor less in $G_k^{(i,j)}$ than it had in $G$, so $\tilde d \leq d \leq \Delta - 1$. Since $q \geq 2 \Delta \geq d + \tilde d + 2$, it suffices to show that $\mathcal P_{G_{k}^{(i,j)}, \vec 1}[\sigma(w_k) = c] \leq 1/(q - d)$ for all $c \in \Gamma_{w_k}$.

    Note that $Z_{G_k^{(i,j)}, w_k}^{(c)}(\vec 1)$ counts the number of proper colorings which are consistent with the partial coloring of $G_k^{(i,j)}$ and that assign $\sigma(w_k) = c$. Hence, we need to prove that the number of colorings in which $\sigma(w_k) = c$ is at most a $1/(q-d)$-fraction of the total. We will do that by showing there exists an injective function $f : \mathcal C^{(c)} \times [q-d] \longrightarrow \mathcal C$, where $\mathcal C$ is the set of all possible colorings of $G_{k}^{(i,j)}$ that agree with its partial coloring, and $\mathcal C^{(c)} \subset \mathcal C$ is the subset of those which assign $\sigma(w_k) = c$.

    Fix an order on the colors (for example, the natural one induced by $[q] \subset \mathbb Z$). For a given $\sigma \in \mathcal C^{(c)}$ and $\ell \in [q-d]$, we define $f(\sigma, \ell)$ as the coloring where $f(\sigma, \ell)(w_k)$ is the $\ell$-th element of $[q] \setminus \sigma(N(w_k))$, while $f(\sigma, \ell)(z) = \sigma(z)$ for all $z \neq w_k$. Note that this is well-defined, as $\left\vert \sigma(N(w_k)) \right \vert \leq q - d$. It just remains to show that this function is injective.

    Assume $\sigma, \tilde \sigma \in \mathcal C^{(c)}$ and $\ell, \tilde \ell \in [q-d]$ satisfy $f(\sigma, \ell) = f(\tilde \sigma, \tilde \ell)$. That means that $\sigma(z) = f(\sigma, \ell)(z) = f(\tilde \sigma, \tilde \ell)(z) = \tilde \sigma(z)$ for all $z \neq w_k$, while $\sigma(w_k) = c = \tilde\sigma(w_k)$, so $\sigma = \tilde \sigma$. Hence, we have the equality $\sigma(N(w_k)) = \tilde \sigma(N(w_k))$, so $f(\sigma, \ell)(w_k) = f(\tilde \sigma, \tilde \ell)(w_k)$ means that $\ell = \tilde \ell$.
\end{proof}

\begin{remark}
    Our niceness condition is equivalent to taking $w = 0$ in the one in \cite{liu2025correlation}. Note that there is nothing special about the vertex $w_k$ or the graph $G_k^{(i,j)}$: the same proof holds for any vertex with degree at most $\Delta -1$.
\end{remark}

For the next lemma, we adopt the notation $a^{(i)}_{G, u}(\vl) \defeq \ln(P_{G, \vl}[\sigma(u) = i])$ for $i \in \Gamma_u$, and define $f(x) := - \ln(1 - e^x)$. Recall that this function is the same as $f_1$ from \autoref{lemma: f}.

\begin{lemma} \label{lemma: contraction} 
    Assume $u$ is a nice vertex in $G$, and that $d_u \leq \Delta - 1$. Then, for any colors $i, j \in \Gamma_u$, there exists a real constant $C_{G, u, i} = C \in [0, \frac{1}{d_u + \nu}]$ so that 

    \begin{align}
        \Big| \Re \left( f(a_{G, u}^{(i)}(\vl)) -  f(a_{G, u}^{(i)}(\vec{1})) \right) - C\cdot \Re \left( a_{G, u}^{(i)}(\vl) - a_{G, u}^{(i)}(\vec{1}) \right)   \Big| & \leq \eps_I; \label{eq:contraction_real_part} \\
        \Big| \Im f(a_{G, u}^{(i)}(\vl)) -  \Im f(a_{G, u}^{(j)}(\vl))  \Big| & \leq \varepsilon_I; \label{eq:contraction_imag_part}
        \\
        \Big| \Im f(a_{G, u}^{(i)}(\vl)) \Big| & \leq \eps_I; \label{eq:contraction_imag_part_only_one_term}
    \end{align}
   
\end{lemma}

\begin{proof}
    We will show that we can define parameters $\zeta$ and $\tau$ satisfying the hypothesis from \autoref{lemma: f}, so that $a_{G, u}^{(i)}(\vec \lambda), a_{G, u}^{(i)}(\vec 1), a_{G, u}^{(j)}(\vec \lambda) \in D_{\zeta, \tau}$, where $D$ is the domain defined in \autoref{lemma: f}. Then, we will be able to apply \autoref{lemma: mean} to obtain the desired bounds.
    \begin{claim}\label{claim:define_parameters_zeta_tau}
        The parameters
        \begin{align*}
            \zeta &:= \ln(d_u + 2) - (d_u + 0.11) \varepsilon_R, \quad \text{and} \\
            \tau &:= (d_u + 0.91) \varepsilon_I
        \end{align*}
        satisfy the hypothesis from \autoref{lemma: f}.
    \end{claim}
    \begin{claimproof}
        It is clear that both $\zeta$ and $\tau$ are positive, since $\ln 2 > \Delta \varepsilon_R > (d_u + 0.11) \varepsilon_R$. Likewise, it is easy to see that $\tau = (d_u + 0.91) \varepsilon_I \leq \Delta \varepsilon_I < 1/2$. It just remains to show that $e^{-\zeta} + \tau^2 < 1$:
        $$
            e^{-\zeta} + \tau^2 \leq \frac{1}{d_u + 2}e^{\Delta \varepsilon_R} + \Delta^2 \varepsilon_I^2 \leq \frac{1}{2} e^{0.01} + 10^{-8} < 1
        $$
    \end{claimproof}

    \begin{claim}
        Let $D = \{z ~|~ \Re(z) \in (-\infty, -\zeta)\text{ and } |\Im(z)| < \tau\}$,
        where $\zeta$ and $\tau$ are defined as in \autoref{claim:define_parameters_zeta_tau}. Then, for any $i \in \Gamma_u$, we have that $a_{G, u}^{(i)}(\vec \lambda) \in D$ and $a_{G, u}^{(i)}(\vec 1) \in D$.
    \end{claim}
    \begin{claimproof}
        Note that we are assuming that $\vec \lambda$ is an arbitrary complex vector satisfying the assumptions from \autoref{assumptions}. These assumptions also hold for $\vec 1$, so it is enough to show the statement for $a_{G, u}^{(i)}(\vec \lambda)$.

        From the first part of \autoref{lemma:approximation-pseudoprobabilities},
        \begin{align*}
            \left \vert \Im \ln \mathcal P_{G, \vec \lambda}[\sigma(u) = i] \right \vert = \left \vert \Im \left( \ln \frac{\mathcal P_{G, \vec \lambda}[\sigma(u) = i]}{\mathcal P_{G, \vec 1}[\sigma(u) = i]} \right) \right \vert \leq (d_u + \nu) \varepsilon_I < \tau,
        \end{align*}
        where we have used that $\mathcal P_{G, \vec 1}[\sigma(u) = i] \in \mathbb R^+$.

        From the second part of \autoref{lemma:approximation-pseudoprobabilities}, together with the niceness of $u$, we obtain that
        \begin{align*}
            \Re \ln \mathcal P_{G, \vec \lambda}[\sigma(u) = i] &\leq \Re \ln \mathcal P_{G, \vec 1}[\sigma(u) = i] + \left \vert \Re \left( \ln \frac{\mathcal P_{G, \vec \lambda}[\sigma(u) = i]}{\mathcal P_{G, \vec 1}[\sigma(u) = i]} \right) \right\vert \\
            &\leq -\ln(d_u + 2) + (d_u + 0.1) \varepsilon_R \\
            &< -\zeta
        \end{align*}
        Hence, $a_{G, u}^{(i)}(\vec \lambda) \in D$.
    \end{claimproof}
    
    Next, we give some bounds for $\rho_R$ and $\rho_I$, as defined in \autoref{lemma: f}:
    \begin{claim}\label{claim:rho}
    Let $\rho_R := \frac{e^{-\zeta}}{1 - e^{-\zeta}}$ and $\rho_I := \frac{\tau \rho_R}{1 - e^{-\zeta}}$. Then, $\rho_R \leq \frac{1}{d_u + 0.91}$ and $\rho_I \leq 3 \eps_I$.
\end{claim}

\begin{claimproof}
    By definition of $\zeta$, we have
    \begin{align*}
        \rho_R &= \frac{1}{e^\zeta - 1} = \frac{1}{(d_u+2)e^{-(d_u + 0.11)\varepsilon_R} - 1} \leq \frac{1}{(d_u+2)(1 - (d_u + 0.11) \varepsilon_R) - 1} \\
        &= \frac{1}{d_u + 1 - (d_u + 2)(d_u + 0.1)\varepsilon_R} \leq \frac{1}{d_u + 1 - 2\Delta^2 \varepsilon_R} = \frac{1}{d_u + 1 - 0.02} \leq \frac{1}{d_u + 0.91}
    \end{align*}
    where we have used the basic inequality $e^{-x} \geq 1 - x$, together with $d_u + 1 \leq \Delta$.
    \\
    We use the bound on $\rho_R$ for bounding $\rho_I$:
    \begin{align*}
        \rho_I \leq \frac{(d_u + 0.91) \varepsilon_I}{d_u + 0.91} \frac{1}{1 - e^{-\zeta}} = \frac{\varepsilon_I}{1 - e^{-\zeta}}
    \end{align*}
    It just remains to show that $e^{-\zeta} \leq 2/3$, which is straightforward:
    \begin{align*}
        e^{-\zeta} = \frac{1}{d_u + 2} e^{(d_u+ 0.11) \varepsilon_R} \leq \frac{1}{2} e^{0.01} \leq \frac{2}{3}.
    \end{align*}
\end{claimproof}

From \autoref{lemma: f}, we know that the function $f$ and the domain $D$ satisfy the hypothesis of \autoref{lemma: mean}. That lemma tells us that there exists a constant $C \in [0, \rho_R]$ (possibly depending on $G, u, i$), such that 
\begin{align} \label{eq:application_of_mean_value_thm}
    \Big|\Re \left(f(a_{G,u}^{(i)}(\vl)) - f(a_{G,u}^{(i)}(\vec{1})) \right) - C \cdot \Re \left(a_{G,u}^{(i)}(\vl) - a_{G,u}^{(i)}(\vec{1}) \right) \Big| &\leq  \rho_I \cdot \left|\Im( a_{G,u}^{(i)}(\vl) - a_{G,u}^{(i)} (\vec 1) ) \right|
\end{align}
Using \autoref{lemma:approximation-pseudoprobabilities} and \autoref{claim:rho}, we can bound the right-hand side of eq.~\eqref{eq:application_of_mean_value_thm} with
\begin{equation*}
    \rho_I \cdot \left|\Im( a_{G,u}^{(i)}(\vl) - a_{G,u}^{(i)} (\vec 1) ) \right| \leq 3 \varepsilon_I (d_u + \nu) \varepsilon_I < \varepsilon_I,
\end{equation*}
which proves eq.~\eqref{eq:contraction_real_part}.

Note that $a_{G, u}^{(i)}(\vec \lambda)$, $a_{G, u}^{(j)}(\vec \lambda)$ and $a_{G, u}^{(i)}(\vec \lambda) - a_{G, u}^{(j)}(\vec \lambda)$ all satisfy $\left \vert \Im z \right \vert \leq (d_u + \nu) \varepsilon_I$, in the first two cases due to \autoref{lemma:approximation-pseudoprobabilities}, and in the third case due to item~\ref{induction-item:imag-part} of \autoref{lemma:induction} and \autoref{assumptions}. Therefore, applying the second part of \autoref{lemma: mean} to the pair $a_{G,u}^{(i)}(\lambda)$, $a_{G,u}^{(j)}(\lambda)$, we get eq.~\eqref{eq:contraction_imag_part}:
\begin{align*}
     |\Im(f(a_{G,u}^{i}(\lambda)) - f(a_{G,u}^{j}(\lambda)))| \leq \rho_R (d_u + \nu) \varepsilon_I \leq \varepsilon_I.
\end{align*}

Finally, applying the second part of \autoref{lemma: mean} to the pair $ a_{G,u}^{(i)}(\vec\lambda) , a_{G,u}^{(i)}(\vec 1)$ and using that $a_{G, u}^{(i)}(\vec 1) \in \mathbb R$, we get:
    \begin{align*}
         |\Im(f(a_{G,u}^{(i)}(\lambda)) - f(a_{G,u}^{(i)}(\vec{1})))| \leq \rho_R \cdot 
          |\Im(a_{G,u}^{(i)}(\lambda) )| \leq \varepsilon_I
    \end{align*}
    which once again follows from \autoref{claim:rho} and \autoref{lemma:approximation-pseudoprobabilities}. This proves eq.~\eqref{eq:contraction_imag_part_only_one_term} and finishes the proof of the lemma. 
\end{proof}

\begin{remark}
    Note that these bounds are slightly tighter than the ones in the analogous lemma from \cite{liu2025correlation}, since we do not need to approximate $f_{\kappa}$ with $f$. 
\end{remark}

\subsection{Proof of Lemma \ref{lemma:induction}} \label{section: induction proof}

\begin{proof}
We will prove \autoref{lemma:induction} by induction on the number of unpinned vertices. Let us first prove it for the base case in which $u$ is the only unpinned vertex:
\begin{claim}
    If $u$ is the only unpinned vertex of $G$, the induction hypothesis from \autoref{lemma:induction} holds.
\end{claim}
\begin{claimproof}
    \begin{itemize}
        \item [(\ref{induction-item:zero-freeness})]
        Note that $Z_{G, u}^{(i)}(\vl)$ sums over the colorings which agree with the partial coloring on $G$ and in which $\sigma(u) = i$. Since $u$ is the only unpinned vertex, there is only one possible such coloring $\sigma$. Hence, we have that $\left \vert Z_{G, u}^{(i)}(\vl) \right\vert = \prod\limits_{v \in V} \left \vert \lambda_{\sigma(v)} \right \vert \geq \lambda_\text{min}^{n}$. Note that this can not vanish, as $\vec \lambda$ satisfies  assumption~\ref{assumption:lambdas-nonzero} from \autoref{assumptions}. 
        \item[(\ref{induction-item:ratio-all-pinned})]
        Since $Z_{G, u}^{i}(\vl) = \lambda_i \prod\limits_{v \in V\setminus u} \lambda_{\sigma(v)} $, then $R^{(i,j)}_{G,u}(\vl) = \frac{\lambda_i \prod_{v \in V\setminus u} \lambda_{\sigma(v)}}{\lambda_j \prod_{v \in V\setminus u} \lambda_\sigma(v)} = \frac{\lambda_i}{\lambda_j}$.
        \item[(\ref{induction-item:log-ratio-well-defined})]
        Using item~\ref{induction-item:ratio-all-pinned} together with assumption~\ref{assumption:lambda-argument-small} from \autoref{assumptions},
        $$
        \left \vert \arg(R_{G, u}^{(i,j)}(\vec \lambda)) \right\vert = \left\vert \arg(\lambda_i / \lambda_j) \right\vert \leq \nu \varepsilon_I, 
        $$
        which is much smaller than $\pi$, so $R_{G, u}^{(i,j)}(\vec \lambda) \notin (-\infty, 0)$.
        \item [(\ref{induction-item:real-part}), (\ref{induction-item:imag-part})]
        Follows directly from item~\ref{induction-item:ratio-all-pinned}:
        \begin{align*}
            \left \vert \Re (\ln R^{(i,j)}_{G,u}(\vec\lambda)) - \ln R^{(i,j)}_{G,u}(\vec 1) \right \vert &= \left \vert \Re \ln \frac{\lambda_i}{\lambda_j} \right \vert, \quad \text{and} \\
            \left \vert \Im \ln R^{(i,j)}_{G,u}(\vec \lambda) \right \vert &= \left \vert \Im \ln \frac{\lambda_i}{\lambda_j} \right \vert.
        \end{align*}
        \item [(\ref{induction-item:ratio-for-bad-colors})] If $i \notin \Gamma_u$, then by definition $Z_{G,u}^{(i)} = 0$, while if $j \in \Gamma_u$, then $Z_{G,u}^{(j)} \neq 0$ by item~\ref{induction-item:zero-freeness}. Hence, $R_{G, u}^{(i,j)}(\vec \lambda)$ is well-defined in this case and it takes value 0. 
    \end{itemize}    
\end{claimproof}

Let us now assume that $G$ has $\ell \geq 2$ unpinned vertices (of which $u$ is one) and that \autoref{lemma:induction} holds for any partially-colored graph with at most $ \ell - 1$ unpinned vertices.  
Items~\ref{induction-item:zero-freeness} and \ref{induction-item:ratio-all-pinned} are straightforward to prove. For item~\ref{induction-item:zero-freeness}, observe that when we pin $u$ to color $i$, we obtain a new graph $G'$ with partition function $Z_{G'}(\vl) = \lambda_i^{\deg (u) -1} Z_{G,u}^{(i)} (\vl)$. The graph $G'$ has one fewer unpinned vertex than $G$, so we can assume \autoref{lemma:induction} holds for it. On the other hand, note that the total number of vertices of $G'$ is at most $\Delta - 1$ more than in $G$. Hence, \autoref{corollary:zero-freeness-partially-colored-graph} tells us that $\left \vert Z_{G'}(\vec\lambda) \right\vert \geq 0.99^{\ell-1} \lambda_\text{min}^{n + \ell(\Delta - 1)} / \lambda_\text{max}^{(\ell-1)\Delta}$. By definition, $\left\vert \lambda_i \right\vert \leq \lambda_\text{max}$, so we conclude that $\left \vert Z_{G, u}^{(i)}(\vec \lambda) \right\vert \geq 0.99^{\ell-1} \lambda_\text{min}^{n + \ell(\Delta - 1)} / \lambda_\text{max}^{\ell \Delta} > 0$.
 
For item~\ref{induction-item:ratio-all-pinned}, note that if all neighbors of $u$ are pinned, then $Z_{G, u}^{(k)}(\vl) = \lambda_k Z_{G-u}(\vl) $ for any $k \in \{i, j\}$ (this is not true in general because an unpinned neighbor of $u$ will have the restriction of not being colored $i$ or $j$). Hence, $R^{(i,j)}_{G,u}(\vl) = \frac{\lambda_i Z_{G - u}(\vl)}{\lambda_j Z_{G - u}(\vl)} = \frac{\lambda_i}{\lambda_j}$.
 
For items~\ref{induction-item:log-ratio-well-defined} through \ref{induction-item:imag-part} we will use the recurrence relation from \autoref{lemma: recurrence}:
\begin{equation} \label{eq:rec_formula_ratios}
    R^{(i,j)}_{G,u}(\vl) = 
    \frac{\lambda_i}{\lambda_j} \prod_{k=1}^{\deg_G(u)} \frac{1 - P_{G_k^{(i,j)},\vl}[\sigma(w_k) = i ]}{1 - P_{G_k^{(i,j)},\vl}[\sigma(w_k) = j ]},
\end{equation} 
where $w_1, \dots, w_{\deg_G(u)}$ denote the neighbors of $u$ in $G$.

We may restrict this product only to the vertices $w_k$ that were unpinned in $G$. That's because if $w_k$ was pinned in $G$ to a certain color $c$, then $c \not\in \{i,j\}$ (recall we are taking $i, j \in \Gamma_u$), so $P_{G_k^{(i,j)},\vl}[\sigma(w_k) = i ] = P_{G_k^{(i,j)},\vl}[\sigma(w_k) = j ] = 0$. That means that the factor corresponding to $w_k$ does not alter the product. We will thus restrict this product to the unpinned neighbors of $u$, and rename them $w_1, \dots, w_{d_u}$. 

For item~\ref{induction-item:log-ratio-well-defined}, we will show a crude upper-bound on $\left\vert \arg R_{G, u}^{(i,j)}(\vec \lambda) \right\vert$ (note that we already know that $R_{G, u}^{(i,j)}(\vec \lambda) \neq 0$ due to item~\ref{induction-item:zero-freeness}). From eq. \eqref{eq:rec_formula_ratios}:
$$
    \left\vert \arg R_{G, u}^{(i,j)}(\vec \lambda) \right\vert \leq \left\vert \arg\frac{\lambda_i}{\lambda_j} \right\vert + \sum_{k = 1}^{d_u} \left\vert \arg(1 - \mathcal P_{G_k^{(i,j)}, \vec \lambda}[\sigma(w_k) = i]) \right\vert + \sum_{k = 1}^{d_u} \left \vert \arg(1 - \mathcal P_{G_k^{(i,j)}, \vec \lambda}[\sigma(w_k) = j]) \right \vert
$$
The first term can be bounded through assumption~\ref{assumption:lambda-argument-small} of \autoref{assumptions}. For the rest, notice that the term involving $ \mathcal P_{G_k^{(i,j)}, \vec \lambda}[\sigma(w_k) = i]$ either vanishes (if $i \notin \Gamma_{w_k}$) or can be bounded with eq. \eqref{eq:contraction_imag_part_only_one_term} from \autoref{lemma: contraction} (if $i \in \Gamma_{w_k}$). The same logic applies to the terms involving color $j$. Thus,
$$
    \left\vert \arg R_{G, u}^{(i,j)}(\vec \lambda) \right\vert \leq \nu\varepsilon_I + 2 d_u \varepsilon_I \leq 2 \Delta \varepsilon_I \leq 0.0002 < \pi,
$$
which finalizes the proof of item~\ref{induction-item:log-ratio-well-defined}. Note that we are able to apply \autoref{lemma: contraction} (and all the other lemmas which are a consequence of \autoref{lemma:induction}) due to the fact that $G_k^{(i,j)}$ has one unpinned vertex fewer than $G$, and hence is covered by the induction hypothesis. 
 
Let us now focus on items \ref{induction-item:real-part} and \ref{induction-item:imag-part}. For ease of notation, let $A(i)$ be the set of unpinned neighbors of $u$ for which $i$ is a good color, and let $B(i)$ be the set of unpinned neighbors of $u$ for which $i$ is a bad color. Let us also denote $G_k \defeq G_k^{(i,j)}$ for brevity.
 
Taking logarithms in equation \eqref{eq:rec_formula_ratios}, we obtain:
\begin{align} 
     -\ln( R^{(i,j)}_{G,u}(\vl)) = -\ln
    \frac{\lambda_i}{\lambda_j} &+  \sum_{w_k \in A(i)} f(a_{G_k,w_k}^{(i)}(\vec\lambda)) - \sum_{w_k \in A(j)} f(a_{G_k,w_k}^{(j)}(\vec\lambda)) \label{lemma: logarithms} \\
    =  -\ln \frac{\lambda_i}{\lambda_j} &+  \sum_{w_k \in A(i) \cap A(j)} \left( f(a_{G_k,w_k}^{(i)}(\vec\lambda)) - f(a_{G_k,w_k}^{(j)}(\vec\lambda)) \right) \nonumber \\
    &+ \sum_{w_k \in A(i) \cap B(j)} f(a_{G_k,w_k}^{(i)}(\vec\lambda)) - \sum_{w_k \in A(j) \cap B(i)} f(a_{G_k,w_k}^{(j)}(\vec\lambda)) \label{lemma: logarithms_grouped_disjointly}
\end{align}

Using equation \eqref{lemma: logarithms_grouped_disjointly} with both $\vl$ and $\vec{1}$, and taking into account that each $w_k$ can only appear in one of the 3 summands and there are at most $d_u$ of them, we get the following bound:

\begin{align}
    \left|\Re(\ln  R^{(i,j)}_{G,u}(\vl) -  \ln R^{(i,j)}_{G,u}(\vec{1}))\right| \leq \left\vert \Re \ln \frac{\lambda_i}{\lambda_j} \right\vert  + d_u \max \left\{  \max_{w_k \in A(i) \cap B(j)}\left\{ \left \vert \Re f(a_{G_k, w_k}^{(i)}(\vl)) - f(a_{G_k, w_k}^{(i)}(\vec{1})) \right\vert \right\}, \right. \nonumber
    \\ 
    \max_{w_k \in B(i) \cap A(j) } \left\{\left\vert \Re f(a_{G_k, w_k}^{(j)}(\vl)) - f(a_{G_k, w_k}^{(j)}(\vec{1})) \right\vert\right\},  \nonumber
    \\
    \left. \max_{v \in A(i) \cap A(j)}\left\{\left\vert \Re \left( f(a_{G_k, w_k}^{(i)}(\vl)) - f(a_{G_k, w_k}^{(i)}(\vec{1})) \right) - \Re \left( f(a_{G_k, w_k}^{(j)}(\vl)) - f(a_{G_k, w_k}^{(j)}(\vec{1})) \right) \right\vert \right\} \right\}
    \label{equation: three maxs}
\end{align}

Recall that $G_k$ has one pinned vertex fewer than $G$, so all the consequences of \autoref{lemma:induction} hold for it. In particular, we know that $w_k$ is nice in $G_k$, due to \autoref{lemma:vk_is_nice}. Also, due to the definition of $G_k$, $\deg_{G_k}(w_k) = \deg_G(w_k) - 1\leq \Delta - 1$. Thus, we can apply equation \eqref{eq:contraction_real_part} from \autoref{lemma: contraction} for the $w_k \in A(i) \cap B(j)$, and we get that 
\[ \left\vert \Re f(a_{G_k, w_k}^{(i)}(\vl)) - f(a_{G_k, w_k}^{(i)}(\vec{1})) \right\vert \leq \frac{1}{d_{w_k} + \nu }\left\vert\Re (a_{G_k, w_k}^{(i)}(\vl)) - a_{G_k, w_k}^{(i)}(\vec{1})  \right\vert + \eps_I \] 
where $d_{w_k}$ is the number of unpinned neighbors of $w_k$. Substituting \autoref{lemma:approximation-pseudoprobabilities}, we obtain:
\[ \left\vert \Re f(a_{G_k, w_k}^{(i)}(\vl)) - f(a_{G_k, w_k}^{(i)}(\vec{1})) \right\vert \leq \left( \frac{d_{w_k} + 0.1}{d_{w_k} + \nu } \right) \eps_R 
+ \eps_I \]

Note that $\varepsilon_I = \varepsilon_R / 100\Delta^2$, and $d_{w_k} + \nu \leq d_{w_k} + 1 \leq \Delta$, so $\varepsilon_I \leq 0.01 \varepsilon_R / (d_{w_k} + \nu)$. Therefore,

\begin{equation}\label{eq:proof_induction_first_and_second_terms}
\left \vert \Re f(a_{G_k, w_k}^{(i)}(\vl)) - f(a_{G_k, w_k}^{(i)}(\vec{1})) \right\vert \leq \left(\frac{d_{w_k} + 0.11}{d_{w_k} + \nu} \right) \varepsilon_R \leq \varepsilon_R
\end{equation}

An analogous bound holds for the $w_k \in A(j) \cap B(i)$.

The only remaining case is $w_k \in A(j) \cap A(i)$. Recall that for any $s \in \Gamma_{G_k, w_k}$, we can use item~\ref{eq:contraction_real_part} of \autoref{lemma: contraction} to show that there exists a $C_s \leq \frac{1}{d_{w_k} + \nu}$ such that

\begin{equation} \label{eq:proof_induction_contraction}
\Big| \Re \left( f(a_{G, w_k}^{(s)}(\vl)) -  f(a_{G, w_k}^{(s)}(\vec{1})) \right) - C_s \cdot \Re \left( a_{G, w_k}^{(s)}(\vl) - a_{G, w_k}^{(s)}(\vec{1}) \right) \Big|  \leq \eps_I \end{equation}

For $w_k \in A(i) \cap A(j)$, both $i, j \in \Gamma_{G_k, w_k}$, so we can apply \eqref{eq:proof_induction_contraction} twice to get that
\begin{align}
    &\left|\left(\Re f(a^{(i)}_{G_k, w_k}(\vl)) - \Re f(a^{(i)}_{G_k, w_k}(\vec{1})) \right) - \left( \Re f(a^{(j)}_{G_k, w_k}(\vl)) - \Re f(a^{(j)}_{G_k, w_k}(\vec{1})) \right) \right| \nonumber \\
     &\hspace{1.5cm} \leq \left|C_{i}\cdot \Re \left( a_{G, u}^{(i)}(\vl) - a_{G, u}^{(i)}(\vec{1}) \right)    - C_{j}\cdot \Re \left( a_{G, u}^{(j)}(\vl) - a_{G, u}^{(j)}(\vec{1}) \right) \right| + 2 \eps_I \nonumber
    \\
    &\hspace{1.5cm} = \left|C_{i} \Re \xi_{i} - C_{j} \Re \xi_{j} \right| +2 \eps_I \label{equation: xi} 
\end{align} 

where $\xi_i := a_{G, u}^{(i)}(\vl) - a_{G, u}^{(i)}(\vec{1})$.

If $\Re \xi_i$ and $\Re \xi_j$ have the same sign, then 
\begin{gather*}
    \left \vert C_i \Re \xi_i - C_j \Re \xi_j \right \vert \leq \left \vert C_{i_{\text{max}}} \Re \xi_{i_{\text{max}}} \right \vert 
\end{gather*}
where $i_\text{max} := \argmax\limits_{t \in \{i,j\}} \left \vert C_t \Re \xi_t \right \vert$. 

In that case, using \autoref{lemma:approximation-pseudoprobabilities}, we have that
\begin{equation} \label{eq:same_sign_xis}
    \left \vert C_i \Re \xi_i - C_j \Re \xi_j \right\vert \leq \frac{1}{d_{w_k} + \nu} (d_{w_k} + 0.1) \varepsilon_R
\end{equation}

If $\Re \xi_i$ and $\Re \xi_j$ have opposite signs, then
\begin{gather}\label{eq:proof_induction_opposite_signs}
    \left\vert C_i \Re \xi_i - C_j \Re \xi_j \right\vert \leq \frac{1}{d_{w_k} + \nu} \left \vert \Re \xi_i - \Re \xi_j \right\vert
\end{gather}
We can rewrite that as
\begin{align*}
    \left \vert \Re \xi_i - \Re \xi_j \right \vert &= \left \vert \Re \ln \left( \frac{P_{G_k, \vec{\lambda}}[\sigma(w_k) = i]}{P_{G_k, \vec{1}}[\sigma(w_k) = i]} \right) - \Re \ln \left( \frac{P_{G_k, \vec{\lambda}}[\sigma(w_k) = j]}{P_{G_k, \vec{1}}[\sigma(w_k) = j]} \right) \right \vert \\
    &= \left \vert \Re \ln \left( \frac{P_{G_k, \vec{\lambda}}[\sigma(w_k) = i]}{P_{G_k, \vec{\lambda}}[\sigma(w_k) = j]} \right) - \Re \ln \left( \frac{P_{G_k, \vec{1}}[\sigma(w_k) = i]}{P_{G_k, \vec{1}}[\sigma(w_k) = j]} \right) \right \vert \\
    &= \left \vert \Re \ln R_{G_k, w_k}^{i,j}(\vec{\lambda}) - \Re \ln R_{G_k, w_k}^{i,j}(\vec{1}) \right \vert \\
    &\leq d_u \varepsilon_R + \Delta \varepsilon_I
\end{align*}

In the last line we have used item \ref{induction-item:real-part} of lemma \ref{lemma:induction} together with assumption \ref{assumption:lambdas-close-together} from definition \ref{assumptions}. This crucially relies on the fact that $G_k$ has one less unpinned vertex than $G$, so we can assume that lemma \ref{lemma:induction} holds for it. We also needed that $i,j \in \Gamma_{w_k}$, but that is a direct consequence of $w_k \in A(i) \cap A(j)$.

Using that $\varepsilon_I = \varepsilon_R/100\Delta^2$, we end up with
\begin{align*}
    \left \vert \Re \xi_i - \Re \xi_j \right \vert \leq \left( d_{w_k} + \frac{1}{100 \Delta} \right) \varepsilon_R
\end{align*}

Plugging this into equation \eqref{eq:proof_induction_opposite_signs}, we see that the bound from equation \eqref{eq:same_sign_xis} also holds in this case. Hence, regardless of whether $\Re \xi_i$ and $\Re \xi_j$ have the same or opposite signs, \begin{equation} \label{eq:proof_induction_bound_diff_real_parts_xis}
    \left \vert C_i \Re \xi_i - C_j \Re \xi_j \right\vert \leq \frac{1}{d_{w_k} + \nu} (d_{w_k} + 0.1) \varepsilon_R \leq \varepsilon_R
\end{equation}
Together with equations \eqref{eq:proof_induction_first_and_second_terms} and \eqref{equation: three maxs}, this finalizes the proof of item \ref{induction-item:real-part}.

For item \ref{induction-item:imag-part}, we look at the imaginary part of equation \eqref{lemma: logarithms_grouped_disjointly}:
\begin{align}
    \left|\Im(\ln  R^{(i,j)}_{G,u}(\vec\lambda))\right| 
    \leq \left|\Im \ln \frac{\lambda_i}{\lambda_j}\right| 
    + d_u\max \Big\{ &\max_{w_k \in A(i) \cap B(j)}\{| \Im f(a_{G_k, w_k}^{(i)}(\vl)) |\}, \nonumber \\
    &\max_{w_k \in B(i) \cap A(j) } | \Im f(a_{G_k, w_k}^{(j)}(\vl)) | ,  \nonumber \\
    & \max_{w_k \in A(i) \cap A(j)}\{| \Im f(a_{G_k, w_k}^{(i)}(\vl)) - \Im f(a_{G_k, w_k}^{(j)}(\vl)) |\}
      \Big\} \label{equation: proof of 4}
\end{align}

Using items~\ref{eq:contraction_imag_part} and \ref{eq:contraction_imag_part_only_one_term} of \autoref{lemma: contraction}, we can upper-bound each of the three terms of the maximum by $\varepsilon_I$. Note that the lemma requires the color to be good for the vertex in question, and indeed $i$ is good for the $w_k$ from the first term (as $w_k \in A(i)$), $j$ is good for the $w_k$ from the second term (as $w_k \in A(j)$), and both $i$ and $j$ are good for the $w_k$ from the third term (as $w_k \in A(i) \cap A(j)$).  
 
Substituting into eq. \eqref{equation: proof of 4}, we get that  
\[  \left|\Im(\ln  R^{(i,j)}_{G,u}(\vec \lambda))\right| 
    \leq 
    \left| \Im  \ln \frac{\lambda_i}{\lambda_j}\right| + d_u \eps_I,\]
which proves item~\ref{induction-item:imag-part}.
 
Lastly, note that item~\ref{induction-item:ratio-for-bad-colors} follows from item~\ref{induction-item:zero-freeness}, as this implies that  $Z_{G,u}^{(j)}(\vec \lambda) \neq 0$, while $Z_{G,u}^{(i)}(\vec \lambda) = 0$ by definition, since $i \notin \Gamma_u$. Hence, $R^{(i,j)}_{G,u}(\vl) = 0$ in that case. 
\end{proof}

\section{Local Central Limit Theorem}\label{section: LCLT}
In this section we prove a Local Central Limit Theorem (LCLT) for the color class sizes in our coloring model. 
The proof is structured in the following way. In Subsection \ref{subsec: moments as derivatives} we show the relation between expectation, covariances and derivatives of $\log Z$, which will be useful throughout the section. In Subsection \ref{subsec: determinant} we show the asymptotics of the eigenvalues and the determinant of the covariance matrix, which is helpful in both applying the LCLT later for sampling and in proving it. In Subsection \ref{subsection: cluster of char} we approximate the logarithm of the characteristic function $\phi \defeq \E[e^{i\langle \vec t, \vec X\rangle}]$ by its first two cumulants. Note that this can only be done in a small range of $ \vec t$ (since we need to use zero-freeness region to guarantee that Taylor series of $\log \phi$  converge). To handle larger values of $\vec t$, in Subsection \ref{subsection: upperbound char} we give an upperbound on the characteristic function. We combine both results to get LCLT in Subsection \ref{subsection: lclt together}. 
 
\textbf{Notation for this section}: 
Let $R$ denote the zero-freeness radius from \autoref{corollary:zero-freeness-radius}. Let $\mathcal D_R := \{\vec \lambda \in \mathbb C^q \, : \, \|\vec \lambda - \vec 1 \|_\infty < R\}$. Recall from \autoref{remark:existence-branch-for-logZ} that we can define an analytic $\log Z(\vec \lambda)$ in $\mathcal D_R$ by taking an appropriate branch of the logarithm. We will denote it as $F(\vec \lambda) := \log Z(\vec \lambda)$.

For a fixed $\vec \lambda \in \mathbb R^q \cap \mathcal D_R$, we define $(X_1, \dots, X_q)$ to be the random vector of color class sizes of a $\sigma \sim \mu_{G, \vl}$ (the Gibbs distribution for the coloring model). Note that these satisfy the linear relation $\sum_{i = 1}^q X_i = n$, so its covariance matrix is not full-rank. We will thus work with the random vector $\vec X  = (X_1, \dots, X_{q-1})$, in which we have dropped the last coordinate, in a slight abuse of notation. We will denote its expected value and covariance matrix as $\mu := \E[\vec X]$ and $\Sigma := \cov(\vec X)$.

First, we will show that $\vec X$ does have a full-rank covariance matrix, and give asymptotics for the eigenvalues and determinant of $\Sigma$.

\detasymp*

This helps us establish the main result of the section:

\lclt*

\subsection{Moments as Derivatives}\label{subsec: moments as derivatives}
First, we will introduce a preliminary lemma that will show the connection between the expected value and the partition function. 
\begin{lemma}\label{lemma:expected_values_are_derivatives}
    Let $\vec \lambda \in \mathbb R^q \cap \mathcal D_R$. Then, for all $i\in [q]$ 
    \[ \lambda_i \frac{\partial \log(Z(\vl))}{\partial \lambda_i} =  \E_{\vl}[X_i]. \]
\end{lemma}

\begin{proof}
    Note that \[ \frac{\partial\log(Z(\vl))}{\partial \lambda_i} = \frac{ \sum_{\substack{\sigma \in [q]^V \\ \sigma \text{ proper}}} |\sigma^{-1}(i)| \lambda_i^{|\sigma^{-1}(i)|-1}  \prod_{j = 1; j \not = i}^q \lambda_j^{|\sigma^{-1}(j)|} }{Z(\vl)},  \]
    hence,
    \[ \lambda_i \frac{\partial\log(Z(\vl))}{\partial \lambda_i} = \frac{ \sum_{\substack{\sigma \in [q]^V \\ \sigma \text{ proper}}} |\sigma^{-1}(i)| \prod_{j = 1}^q \lambda_j^{|\sigma^{-1}(j)|} }{Z(\vl)} = \sum_{\sigma} |\sigma^{-1}(i)|  \P[\sigma] = \E_{\vl}[X_i], \]
    where $\P[\sigma]$ is the probability of a coloring $\sigma$ to be sampled, and we will sometimes use the subscript for $\E_{\vl}[X_i]$ to indicate the probability distribution depends on $\vl$. 
\end{proof}

\begin{lemma}\label{lemma:covariance_is_double_derivative}
Let $\vec \lambda \in \mathbb R^q \cap \mathcal D_R$. Then, for all
$i,j\in[q]$, we have
\[
\Cov_{\vec{\lambda}}(X_i,X_j)
=
\lambda_i \lambda_j
\frac{\partial^2}{\partial \lambda_i\,\partial \lambda_j}
\log Z(\vec{\lambda})
+
\delta_{ij}\,
\lambda_i
\frac{\partial}{\partial \lambda_i}
\log Z(\vec{\lambda}),
\]
where $\delta_{ij}$ denotes the Kronecker delta.
\end{lemma}

\begin{proof}
By \autoref{lemma:expected_values_are_derivatives},
\[
\E_{\vec{\lambda}}[X_i]
=
\lambda_i
\frac{\partial}{\partial \lambda_i}
\log Z(\vec{\lambda}).
\]
Differentiating both sides with respect to $\lambda_j$ gives
\begin{equation}\label{eq:diff_expectation}
\frac{\partial}{\partial \lambda_j}
\E_{\vec{\lambda}}[X_i]
=
\delta_{ij}
\frac{\partial}{\partial \lambda_i}
\log Z(\vec{\lambda})
+
\lambda_i
\frac{\partial^2}{\partial \lambda_i\,\partial \lambda_j}
\log Z(\vec{\lambda})
\end{equation}

On the other hand,
\[
\E_{\vec{\lambda}}[X_i]
=
\sum_{\sigma}
X_i(\sigma)\,
\P_{\vec{\lambda}}(\sigma),
\qquad
\P_{\vec{\lambda}}(\sigma)
=
\frac{1}{Z(\vec{\lambda})}
\prod_{k=1}^q \lambda_k^{X_k(\sigma)}.
\]

Then,
\[
\frac{\partial}{\partial \lambda_j}\P_{\vec{\lambda}}(\sigma)
=
\frac{1}{Z(\vec{\lambda})}
\frac{\partial}{\partial \lambda_j}
\left(
\prod_{k=1}^q \lambda_k^{X_k(\sigma)}
\right)
-
\frac{1}{Z(\vec{\lambda})^2}
\frac{\partial Z(\vec{\lambda})}{\partial \lambda_j}
\prod_{k=1}^q \lambda_k^{X_k(\sigma)}.
\]
Since
\[
\frac{\partial}{\partial \lambda_j}
\left(
\prod_{k=1}^q \lambda_k^{X_k(\sigma)}
\right)
=
\frac{X_j(\sigma)}{\lambda_j}
\prod_{k=1}^q \lambda_k^{X_k(\sigma)},
\]
and
\[
\frac{\partial Z(\vec{\lambda})}{\partial \lambda_j}
=
\frac{1}{\lambda_j}
\sum_{\sigma'}
X_j(\sigma')
\prod_{k=1}^q \lambda_k^{X_k(\sigma')},
\]
we obtain
\[
\frac{\partial}{\partial \lambda_j}
\P_{\vec{\lambda}}(\sigma)
=
\frac{1}{\lambda_j}
\bigl(
X_j(\sigma)
-
\E_{\vec{\lambda}}[X_j]
\bigr)
\P_{\vec{\lambda}}(\sigma).
\]

Therefore,
\[\frac{\partial}{\partial \lambda_j}
\E_{\vec{\lambda}}[X_i]
=
\sum_{\sigma} X_i(\sigma)\frac{\partial}{\partial \lambda_j} \P_{\vec{\lambda}}(\sigma)
=
\frac{1}{\lambda_j}
\left(\E_{\vec{\lambda}}[X_i X_j]
-
\E_{\vec{\lambda}}[X_i]\E_{\vec{\lambda}}[X_j]\right)
=
\frac{1}{\lambda_j}\Cov_{\vec{\lambda}}(X_i,X_j).\]

Multiplying by $\lambda_j$ and comparing with
\eqref{eq:diff_expectation}, we conclude that
\[\Cov_{\vec{\lambda}}(X_i,X_j)
=
\lambda_i \lambda_j
\frac{\partial^2}{\partial \lambda_i\,\partial \lambda_j}
\log Z(\vec{\lambda})
+\delta_{ij}\,
\lambda_i \frac{\partial}{\partial \lambda_i}
\log Z(\vec{\lambda}). \]
\end{proof}

\subsection{Determinant Asymptotics}\label{subsec: determinant}
In this subsection we will prove \autoref{det asymp}. 
 
First, we derive the lower bound on the variance of $Y = \sum_{i=1}^{q-1} c_i X_i$. The key idea is to use the law of total variance, conditioning on the colors of a subset of vertices that effectively ``separate'' the rest of the graph into a linear number of connected components, in a similar way as it was done in \cite{jain2022approximate}. It requires the following simple lemma: 

\begin{lemma}[(Separation lemma, (Lemma 3.4 from \cite{jain2022approximate})]\label{lemma: separation}
    In a graph $G$ with maximum degree $\Delta \geq 1$ and $n$ vertices there exists a set $S$ of size $\Omega(n/\Delta^3)$ such that every pair of vertices in it is at distance at least 4. 
\end{lemma}

\begin{proof}
    Note that if we greedily choose an available vertex and remove it together with all vertices at distance 3 or less from it (which means removing a maximum of $\mathcal O(\Delta^3)$ vertices), then we get a set of vertices $S$ of size $\left \vert S \right\vert = \Omega(n/ \Delta^3)$. 
\end{proof}

Next we show that $\sum_{i = 1}^q X_i = n$ is, in a precise sense, the only linear dependence between the random variables $X_i$, and that any other linear combination will have large variance.
\begin{lemma}[(Variance of $\vec{c}\cdot \vec X$)]\label{lemma:variance_of_linear_combination_color_class_sizes}
Let $q \geq 2\Delta$ and $\Delta \geq 1$. For any constants $c_1, \dots, c_{q} \in \mathbb R$ with $\max\limits_{i,j} \left \vert c_i - c_j \right\vert = \Omega_{q, \Delta}(1)$, define $Y \defeq \sum_{i=1}^{q} c_i X_i$. Then,
    \[\Var(Y) = \Omega_{\vl, q, \Delta}(n).\]
\end{lemma}
\begin{proof}
    By \autoref{lemma: separation}, we can find a set $S$ of size $s = \Omega(n/\Delta^3)$ such that each of its vertices is at distance at least 4 from the others. Define $T$ to be the set of vertices at distance $2$ or more from $S$. Note that $G[V \setminus T]$ is split into $s$ connected components $V_j := \{v_j\} \cup N(v_j)$, where $v_j \in S$.
    
    Let $\sigma \sim \mu_{G, \vl}$ (we will drop $\vl$ in the subscripts where there is no confusion). We can think of $\sigma$ as if generated in a two-step process: first we pick $\sigma_T := \sigma\vert_{T}$ (distributed according to $\mu_{G[T]}$), and then pick $\sigma_{V_j} := \sigma \vert_{V_j}$ for each $V_j$ in a conditional way, so that $\sigma \sim \mu_G$ in total.

    By the law of total variance:
    \begin{align*}
        \Var(Y) = \Var\left(\sum c_i X_i\right) &= \E\left[\Var\left(\sum c_i X_i \, \Big| \, \sigma_T\right)\right] + \Var\left(\E\left[\sum c_i X_i\, \Big| \, \sigma_T\right]\right) \\ 
        &\geq \E\left[\Var\left(\sum c_i X_i \, \Big| \, \sigma_T\right)\right] .
    \end{align*}

    Let $Z = (Z_1, \dots, Z_q)$ be the color class sizes of $\sigma_{T}$ and, for each $j \in [s]$, let $Q^i = (Q_1^j, \dots, Q_{q}^j)$ be the color class sizes of $\sigma_{V_j}$. The key point is that, once we condition on $\sigma_T$, the $\{Q^j\}_{j \in [s]}$ are mutually independent, since all the $V_j$ are pair-wise separated in the graph by vertices of $T$. That allows us to split the variance into terms depending only on each of these connected components:
    \begin{align*}
        \E\left[\Var\left(\sum_{i = 1}^q c_i X_i ~\Big|~ \sigma_T\right)\right] = \sum_{j = 1}^{s} \E\left[\Var\left(\sum_{i = 1}^q c_iQ_i^j ~\Big|~ \sigma_T\right)\right].
    \end{align*}
    Since $s = \Omega_{\Delta}(n)$, if we show that each of the terms of the sum is $\Omega_{\vec \lambda, q, \Delta}(1)$, that will imply that $\Var(Y) = \Omega_{\vec \lambda, q, \Delta}(n)$.
    \begin{claim}
        For any $j \in [s]$, $\E\left[\Var\left(\sum_{i=1}^{q}c_i Q_i^j ~\big|~ \sigma_T\right)\right] = \Omega_{\vec \lambda, q, \Delta}(1)$.
    \end{claim}
    \begin{claimproof}
        Let $s, t \in [q]$ be the indices maximizing $\left \vert c_s - c_t \right\vert$. Given $\sigma_T$, consider the following coloring $\tau$ of $V_j$:
        \begin{itemize}
            \item For each $w \in N(v_j)$, set $\tau(w)$ as the first color not appearing in $\sigma_T(N(w) \cap T)$ and different from both $s$ and $t$.
            \item Set $\tau(v_j) = s$.
        \end{itemize}
        On the other hand, consider the alternative coloring $\tilde \tau$, defined as:
        \begin{itemize}
            \item For each $w \in N(v_j)$, set $\tilde \tau(w) = \tau(w)$.
            \item Set $\tilde \tau(v_j) = t$.
        \end{itemize}
        It is clear that $\P[\sigma_{V_j} = \tau \, \vert\, \sigma_T], \P[\sigma_{V_j} = \tilde\tau \, \vert \, \sigma_T] = \Omega_{\vec \lambda, q, \Delta}(1)$, as there are at most $q^{\Delta + 1}$ different colorings of $V_j$ (and their weights are bounded above and below when fixing the minimum and maximum coordinate of $\vec \lambda$). At the same time, note that the difference in $\sum_i c_i Q_i^j$ between the case $\sigma_{V_j} = \tau$ and $\sigma_{V_j} = \tilde \tau$ is exactly $c_k$, which is $\Omega_q(1)$, so the variance will be at least constant as well.
    \end{claimproof}
\end{proof}

The following two corollaries particularize the previous result for the two cases in which we will apply it. Firstly, we will apply it for linear combinations of $\vec X = (X_1, \dots, X_{q-1})$ (which corresponds to setting $c_q = 0$ in the previous lemma) and, secondly, for coefficients $\vec c$ which are orthogonal to $\vec 1$.

\begin{corollary}\label{corollary:variance_lowerbound_for_q-1_coordinates} Let $q \geq 2\Delta$ and $\Delta \geq 1$. For any $\vec c \in \mathbb R^{q-1}$ with $\sum_{i = 1}^{q-1} c_i^2 = 1$, define $Y \defeq \sum_{i=1}^{q-1} c_i X_i$. Then,
    \[\Var(Y) = \Omega_{\vl, q, \Delta}(n).\]
\end{corollary}
\begin{proof}
    Since the sum of squares is 1, there exists an $i \in [q-1]$ such that $\left \vert c_i \right\vert \geq 1/\sqrt q = \Omega_q(1)$. Hence, we can apply \autoref{lemma:variance_of_linear_combination_color_class_sizes} for $\tilde c \in \mathbb R^q$ defined so that $\tilde c_q := 0$ and $\tilde c_j := c_j$ for all $j \in [q-1]$.
\end{proof}

\begin{corollary}\label{corollary:variance_lowerbound_for_q_coordinates}
Let $q \geq 2\Delta$ and $\Delta \geq 1$. For any $\vec c \in \mathbb R^q$ with $\sum_{i=1}^q c_i^2 = 1$ and $\sum_{i = 1}^q c_i = 0$, define  $Y \defeq \sum_{i=1}^{q} c_i X_i$.
Then,
\[
\Var(Y) = \Omega_{\vec \lambda, q, \Delta}(n).
\]
\end{corollary}

\begin{proof}
 Since the sum of squares is $1$, there exists an $s \in [q]$ such that $\left\vert c_s \right\vert \geq 1/\sqrt q$. On the other hand, since $\sum c_i = 0$, there must exist a $t \in [q]$ such that $c_t$ has opposite sign to $c_s$. Therefore, $\max\limits_{i,j} \left \vert c_i - c_j \right \vert \geq 1/\sqrt q = \Omega_q(1)$, so we can apply \autoref{lemma:variance_of_linear_combination_color_class_sizes}.
\end{proof}

This result implies a lower bound on the eigenvalues of the covariance matrix, and therefore on its determinant.

\begin{lemma}[Lower bound on determinant]\label{lower bound det}
Let $q \geq 2\Delta$ and $\Delta \geq 1$.  
Let $\{\nu_i\}_{i\in[q-1]}$ be the eigenvalues of $\Sigma$. Then, $\nu_i = \Omega(n)~~\forall i \in [q-1]$, and hence
    \[ \det(\Sigma) = \Omega(n^{q-1}).\]
\end{lemma}
\begin{proof}
    Assume for contradiction that $\nu = o(n)$ for a certain eigenvalue $\nu$ of $\Sigma$, which we recall is positive semidefinite. Let $h$ be a normalized eigenvector with eigenvalue $\nu$. Then,
    \begin{align*}
    \nu = h^T \Sigma h =  \sum_{1 \leq i,j \leq q-1} \cov(h_i X_i, h_j X_j) = \var(\sum_{i=1}^{q-1} h_i X_i).
    \end{align*}
    Therefore, $\var(\sum_{i=1}^{q-1} h_i X_i) = o(n)$, which contradicts \autoref{corollary:variance_lowerbound_for_q-1_coordinates}. 
\end{proof}
To show the upper bound on the determinant of the covariance matrix, we will prove a general result, which uses zero-freeness to show that the derivatives of the logarithm of the partition function are $O(n)$. First, we will show a simple linear upper bound on the logarithm of the partition function. 

\begin{lemma}\label{lemma: log Z is linear}
For all $\vl \in \mathcal D_R$, we have that
    \[ \left\vert \log Z(\vl) \right\vert = O_{\vec \lambda, q,\Delta}(n)\] 
\end{lemma}
\begin{proof}
    If $\vec \lambda \in \mathcal D_R$, then
    $$
        \left\vert Z(\vec \lambda) \right \vert \leq q^n (1 + R) ^n,
    $$
    while we also have the lower bound
    $$
        \left \vert Z(\vec \lambda) \right \vert \geq 0.99^n \left(\frac{\lambda_\text{min}}{\lambda_{\text{max}}} \right)^{n \Delta},
    $$
    given by \autoref{corollary:lower-bound-norm-partition-function}. Therefore, $\left \vert \log \left(\vert Z(\vec \lambda) \vert \right) \right \vert = O_{\vec \lambda, q, \Delta}(n)$, and hence $\left \vert \log Z(\vec \lambda) \right \vert $ is bounded as well (as they only differ by a constant factor depending on the argument).
\end{proof}

\begin{lemma}[(Linear upper bounds for derivatives of $\log Z$)]\label{lem:cauchy_bounds_logZ}
Let $\vec \zeta\in\mathbb C^q$ be such that $\|\vec \zeta-\vec{1}\|_\infty\le \frac{3R}{4}$.
Then, for every integer $k\ge1$ and every choice of indices $i_1,\dots,i_k\in[q]$, 
the partial derivatives of $F$ satisfy
\[\bigl|\partial_{i_1}\cdots\partial_{i_k}F(\vec \zeta)\bigr|\le
\left(\frac{4}{R}\right)^k\sup_{\|\vec \xi-\vec{1}\|_\infty < R}|F(\vec \xi)|.\]
This also implies (using \autoref{lemma: log Z is linear}) that $\bigl| \partial_{i_1}\cdots\partial_{i_k}F(\vec \zeta)
\bigr| = O_{\vec \lambda, q, \Delta}(n)$.
\end{lemma}

\begin{proof}
We will show that, for any $\rho < R/4$,
$$
    \bigl|\partial_{i_1}\cdots\partial_{i_k}F(\vec \zeta)\bigr|\le
    \frac{1}{\rho^k}\sup_{\|\vec \xi-\vec{1}\|_\infty < R}|F(\vec \xi)|,
$$
from which the claimed result follows by taking the limit.

Recall that $F$ is analytic on $\mathcal D_R$. Note as well that the closed polydisc
\[
\mathcal P(\zeta,\rho)
:=
\{\xi\in\mathbb C^q:\|\xi-\zeta\|_\infty\le\rho\}
\]
is contained in $\mathcal D_R$.
Define $\xi^{(i_1,\dots,i_k)} \in\mathcal P(\zeta,\rho)$ such that only the coordinates $\xi_{i_1},\dots,\xi_{i_k}$ vary and all other coordinates are equal to $\zeta$.

By the multivariate Cauchy integral formula, we have
\[
\partial_{i_1}\cdots\partial_{i_k}F(\zeta)
=
\frac{1}{(2\pi i)^k}
\int_{|\xi_{i_1}-\zeta_{i_1}|=\rho}
\cdots
\int_{|\xi_{i_k}-\zeta_{i_k}|=\rho}
\frac{
F(\xi^{(i_1,\dots,i_k)})
}{
\prod_{m=1}^k (\xi_{i_m}-\zeta_{i_m})^2
}
\,
d\xi_{i_k}\cdots d\xi_{i_1}.
\]

Taking absolute values and using
\(|\xi_{i_m}-\zeta_{i_m}|=\rho\),
\[
\bigl|
\partial_{i_1}\cdots\partial_{i_k}F(\zeta)
\bigr|
\le
\frac{1}{(2\pi)^k}
\int_{|\xi_{i_1}-\zeta_{i_1}|=\rho}
\cdots
\int_{|\xi_{i_k}-\zeta_{i_k}|=\rho}
\frac{
|F(\xi^{(i_1,\dots,i_k)})|
}{
\rho^{2k}
}
\,
|d\xi_{i_k}|\cdots |d\xi_{i_1}|.
\]

Each contour has length \(2\pi\rho\), hence
\[
\bigl|
\partial_{i_1}\cdots\partial_{i_k}F(\zeta)
\bigr|
\le
\frac{(2\pi\rho)^k}{(2\pi)^k}
\cdot
\frac{1}{\rho^{2k}}
\sup_{\xi\in\mathcal P(\zeta,\rho)}|F(\xi)|
=
\frac{1}{\rho^k}
\sup_{\xi\in\mathcal P(\zeta,\rho)}|F(\xi)|.
\]

Since $\mathcal P(\zeta,\rho)\subset\mathcal D_R$,
\[
\sup_{\xi\in\mathcal P(\zeta,\rho)}|F(\xi)|
\le
\sup_{\|\xi-\vec{1}\|_\infty < R}|F(\xi)|.
\]
Combining the above bounds completes the proof.
\end{proof}

As a simple corollary, we get that expectation and covariance are also at most linear. 

\begin{corollary}\label{cor: expect and cov are O(n)}
For $\vl \in \mathcal D_R$, 
    \[ \E_{\vl}[X_i] = O_{\vec \lambda, q, \Delta}(n), \quad \Cov_{\vl}(X_i,X_j) = O_{\vec \lambda, q, \Delta}(n) \]
\end{corollary}

\begin{proof}
    From \autoref{lemma:expected_values_are_derivatives} 
\[  \E_{\vl}[X_i] = \lambda_i \frac{\partial F(\vl)}{\partial \lambda_i}  \] and 
from \autoref{lemma:covariance_is_double_derivative}
\[\Cov_{\vl}(X_i,X_j)
=\lambda_i\lambda_j\,\partial_{ij}F(\vl)
+\delta_{ij}\lambda_i\,\partial_i F(\vl),
\]
thus we get the linear upper bound. 
\end{proof}

Having the upper bounds on the entries of the covariance matrix, it is simple to upper bound its determinant. 
\begin{corollary}[(Upper bound on determinant of covariance matrix)]\label{cor: up bound on det}
Fix a $\vl \in \mathcal D_R$, and let $\nu_{\text{max}}$ be the maximum eigenvalue of $\Sigma$. Then,
    \[
        \nu_{\text{max}} = O_{\vl, q, \Delta}(n),
    \]
    and therefore
    \[ \det(\Sigma) = O_{\vec \lambda, q, \Delta}( n^{q-1}).\]
\end{corollary}
\begin{proof}
    From \autoref{cor: expect and cov are O(n)}, we know that there exists a constant $C = C(\vec \lambda, q, \Delta)$ such that $\left\vert \Sigma_{ij} \right\vert \leq Cn$ for all $i,j \in [q-1]$. Since $\Sigma$ is positive semidefinite, its eigenvalues satisfy 
    \[
        0 \leq \nu_i \leq \operatorname{tr} \Sigma \leq (q-1)Cn.
    \]
    This implies as well that
    \begin{align*}
        0 \leq \det \Sigma \leq ((q-1)Cn)^{q-1}
    \end{align*}
\end{proof}
Putting together \autoref{lower bound det} and \autoref{cor: up bound on det}, we get that $\det(\Sigma) = \Theta( n^{q-1})$.

\subsection{Taylor Expansion of the Characteristic Function}\label{subsection: cluster of char}
In this subsection we analyze the logarithm of the characteristic function of $\vec X$ (that is, $\phi(t) \defeq \E[e^{i\langle t, \vec X\rangle}]$) and show that, for small $t$, the first two cumulants dominate. Note that $\phi : \mathbb R^{q-1} \longrightarrow\mathbb C$, since we are using the notation $\vec X = (X_1, \dots, X_{q-1})$. We will use the notation $\tilde \phi$ for the characteristic function of the full vector $(X_1, \dots, X_q)$, which now takes as input $t \in \mathbb R^q$.
 
Let us first observe that we can rewrite the characteristic function in terms of the partition function:
\begin{observation}\label{obs: char function part} 
For $t\in\mathbb R^q$
    \[\tilde \phi(t) = \frac{Z_G(\lambda_1 e^{it_1},\dots,\lambda_q e^{it_q})}{Z_G(\vec{\lambda})}\]
\end{observation}
Next, we show that if $\vec \lambda$ is not close to the boundary of the zero-freeness region and $t$ is small enough, then the point $(\lambda_1e^{it_1}, \dots, \lambda_q e^{it_q})$ will still be in the zero-freeness region, so $Z_G(\lambda_1 e^{it_1},\dots,\lambda_q e^{it_q}) \neq 0$ and hence $\log \tilde \phi(t)$ will be analytic. 
\begin{lemma}\label{zero-freeness for char fun with small t}
Given a constant $C > 0$, let $\vec{\lambda}\in\mathbb{C}^q$ satisfy $\|\vec{\lambda} - \vec{1}\|_{\infty} < \frac{C}{2}$. Then, for every $t\in\mathbb{C}$ with
$|t|\le \frac{C}{4 + 2C}$ 
we have
\[    \bigl|\lambda_j e^{i t} - 1\bigr| < \frac{3C}{4}
    \qquad \text{for all } j\in[q].\]
\end{lemma}

\begin{proof}
Note that
\begin{align*}
    \bigl|\lambda_j e^{i t} - 1\bigr|
    \le
    |\lambda_j - 1|
    +
    |\lambda_j|\,|e^{i t} - 1|. 
\end{align*} 

Combining that $|\lambda_j|  \le |\lambda_j - 1| + 1$ with the inequality $ |e^{iu} - 1| \le |u|$, which holds for any $u \in \mathbb R$, we obtain
\[   \bigl|\lambda_j e^{i t} - 1\bigr|  < \frac{C}{2} +\Bigl(1+\frac{C}{2}\Bigr)|t|.\]
If $|t|\le \frac{C}{4 + 2C}$, then
\[    \bigl|\lambda_j e^{i t} - 1\bigr| < \frac{3C}{4}.\]
\end{proof}

\begin{lemma}[(Third-order Taylor remainder bound for $\log Z$)]
\label{lem:taylor_remainder_logZ_lambda}
Let $\vec \lambda, \vec \lambda' \in \mathbb C^q$ satisfy $\|\vec \lambda - \vec 1\|_\infty < \frac{3R}{4}$ and $\|\vec \lambda' - \vec 1 \|_\infty < \frac{3R}{4}$. Then $F$ admits the Taylor expansion
\[F(\vec\lambda')
=
F(\vl)
+
\sum_{j=1}^q \partial_j F(\vl)\,(\lambda'_j-\lambda_j)
+
\frac12
\sum_{j,k=1}^q
\partial_{jk}F(\vl)
(\lambda'_j-\lambda_j)(\lambda'_k-\lambda_k)
+
R_3,
\]
where the remainder satisfies
\[
|R_3|
\le
C\,n
\sum_{j,k,\ell=1}^q
|\lambda'_j-\lambda_j|\,
|\lambda'_k-\lambda_k|\,
|\lambda'_\ell-\lambda_\ell|,
\]
for some constant $C=C(R, q, \Delta)$.
\end{lemma}

\begin{proof}
Let $\vec h:=\vec\lambda'-\vl$. Taylor's theorem with integral remainder gives the claimed expansion with remainder
\[R_3 = \sum_{j,k,\ell} h_j h_k h_\ell
\int_0^1 \frac{(1-s)^2}{2}\, \partial_{jk\ell}F(\vl+s \vec h)\,ds.\]
For $s \in [0,1]$, the convexity of the $\ell_\infty$ ball implies that $\|\vec \lambda + s \vec h\|_\infty < \frac{3R}{4}$. Hence, by \autoref{lem:cauchy_bounds_logZ}, there exists a $C'(R, q, \Delta)$ such that  
\[|R_3|\le \sum_{j,k,\ell} |h_j h_k h_\ell| \int_0^1 \frac{(1-s)^2}{2} C' n \, ds
=
\frac{1}{6} C'n \sum_{j,k,\ell}
|\lambda'_j-\lambda_j|\,
|\lambda'_k-\lambda_k|\,
|\lambda'_\ell-\lambda_\ell|.\]
\end{proof}

\begin{lemma}[(Expansion of the characteristic function)]\label{lem:log_phi_expansion}
Let $\vl, \vec t \in \mathbb R^q$ such that $\|\vl-\vec{1}\|_\infty < \frac{R}{2}$ and
$\|t\|_\infty\le \frac{R}{4 + 2R}$. Then,
\[\log\tilde \varphi(t)
=
i\langle t,\mathbb E_{\vl} \vec X\rangle
-\frac12\, t^\top \mathrm{Cov}_{\vl}(\vec X)\, t
+R_3(t),
\]
where $\vec X = (X_1, \dots, X_q)$ and
\[
|R_3(t)|\le C\,n\,\|t\|_\infty^3,
\]
for some constant $C=C(R,q, \Delta)$.
\end{lemma}

\begin{proof}
Let $\vec\lambda' := (\lambda_1 e^{it_1}, \dots, \lambda_q e^{it_q})$. Recall from \autoref{obs: char function part} that
\[\log\tilde\varphi(t) = F(\vec \lambda')-F(\vec \lambda),\]

By \autoref{zero-freeness for char fun with small t}, $\|\vec \lambda' - \vec 1\|_\infty < \frac{3R}{4}$. Hence,  \autoref{lem:taylor_remainder_logZ_lambda} implies that
\begin{align*}
F(\vec \lambda')
= F(\vl)
&+ \sum_{j=1}^q \partial_j F(\vl)\,
(\lambda_j e^{it_j}-\lambda_j)\\
&
+
\frac12\sum_{j,k=1}^q \partial_{jk}F(\vl)
(\lambda_j e^{it_j}-\lambda_j)
(\lambda_k e^{it_k}-\lambda_k) +R_3,
\end{align*}
where
\[|R_3| \le
C\,n \sum_{j,k,\ell = 1}^q |\lambda_j(e^{it_j}-1)|
|\lambda_k(e^{it_k}-1)| |\lambda_\ell(e^{it_\ell}-1)|. \]

For each $j \in [q]$, there exist $z_{j,1}, z_{j,2}, z_{j,3} \in \mathbb C$ with $\left \vert z_{j,m} \right\vert \leq \left \vert t_j \right\vert^m$ such that 
\[e^{it_j}-1 = z_{j,1} = it_j + z_{j, 2} =  i t_j -\frac12 t_j^2 + z_{j, 3}.\]

Substituting that into the expressions above, we can recognize some of the terms to be expectation and covariances. For example, recalling \autoref{lemma:expected_values_are_derivatives}, we obtain
\[\sum_j \partial_j F(\vl)\,
(\lambda_j e^{it_j}-\lambda_j)
= \E_{\vl}[X_j] \left(it_j - \frac 1 2 t_j^2 + z_{j,3} \right).\]

Using the covariance identities from \autoref{lemma:covariance_is_double_derivative}:
\[\lambda_j\lambda_k\partial_{jk}F(\vl)
= \begin{cases}
    \Cov_{\vl}(X_j,X_k),\quad &\text{if $j \neq k$} \\\\
    \Var_{\vl}(X_j)-\mathbb E_{\vl}[X_j], \quad &\text{if $j = k$}
\end{cases}
\]
we obtain
\begin{align*}
\frac12\sum_{j,k}
\partial_{jk}F(\vl)
&(\lambda_j e^{it_j}-\lambda_j)
(\lambda_k e^{it_k}-\lambda_k)\\
&= \frac{1}{2} \sum_{j, k} \cov_{\vl}(X_j, X_k) (it_j + z_{j,2})(it_k + z_{k,2}) - \frac{1}{2} \sum_{j} \E_{\vl}[X_j](it_j + z_{j,2})^2 \\
&=
-\frac12\sum_{j,k}
t_j t_k\,\mathrm{Cov}_{\vl}(X_j,X_k)
+\frac12\sum_j t_j^2\,\mathbb E_{\vl}[X_j]
+O(n\|t\|_\infty^3),
\end{align*}
where we have used the linear bounds on the expected value and covariance from \autoref{cor: expect and cov are O(n)}.

Finally, the Taylor remainder satisfies
\[|R_3| \le C\,n\sum_{j,k,\ell}\left(1 + \frac{R}{2}\right)^3|z_{j,1}z_{k,1}z_{\ell, 1}|
= O(n \|t\|_\infty^3).\]

Combining all terms,
\[\log\tilde \varphi(t)
=
i\langle t,\mathbb E_{\vl} \vec X\rangle
-\frac12\, t^\top \mathrm{Cov}_{\vl}(\vec X)\, t
+O(n\|t\|_\infty^3).\]
\end{proof}

The lemma above can be adapted straightforwardly to the characteristic function with $q-1$ coordinates:
\begin{corollary}\label{corollary:log_phi_expansion_with_q-1_coordinates}
    Let $\vl \in \mathbb R^q$ and $\vec t \in \mathbb R^{q-1}$ such that $\|\vl-\vec{1}\|_\infty < \frac{R}{2}$ and
    $\|t\|_\infty\le \frac{R}{4 + 2R}$. Then,
    \[
        \log \varphi(t) = i\langle t,\mathbb E_{\vl} \vec X\rangle -\frac12\, t^\top \mathrm{Cov}_{\vl}(\vec X)\, t+R_3(t),
    \]
    where $\vec X = (X_1, \dots, X_{q-1})$ and 
    \[
        |R_3(t)|\le C\,n\,\|t\|_\infty^3,
    \]
    for some constant $C=C(R,q, \Delta)$.
\end{corollary}
\begin{proof}
    Apply \autoref{lem:log_phi_expansion} with $\tilde t = (t_1, \dots, t_{q-1}, 0) \in \mathbb R^q$, noting that $\varphi(t) = \tilde \varphi(\tilde t)$.
\end{proof}

\subsection{Upper Bounds on the Characteristic Function}\label{subsection: upperbound char}
In this subsection we will use an approach of Jain, Perkins, Sah, and Sawhney \cite{jain2022approximate} to upper bound the characteristic function. We use their structure so the reader can see the differences if desired. Note that in our case we have to deal with a multidimensional object. 
\begin{lemma}\label{lemma: characteristic upperbound}
    For any $q\geq \max\{2\Delta, 3\}$, $\Delta \geq 1$ 
    there exists a constant $c = c(\Delta, q, \vl)$ with the following property. 
    For all $t \in [- \pi , \pi]^{q-1}$ the following holds:
    \[ |\E[e^{i{\langle}t, \vec X{\rangle}}]| \leq \exp(-c n \|t\|^2 )  \]
\end{lemma}

\begin{proof}
    We use the same conditioning strategy as in \autoref{lemma:variance_of_linear_combination_color_class_sizes}. By \autoref{lemma: separation}, we can find a set $S$ of size $s = \Omega(n/\Delta^3)$ such that each of its vertices is at distance at least 4 from the others. Define $T$ to be the set of vertices at distance $2$ or more from $S$. Note that $G[V \setminus T]$ is split into $s$ connected components $V_j := \{v_j\} \cup N(v_j)$, where $v_j \in S$.
    
    We can think of a $\sigma \sim \mu_{G}$ as if generated in a two-step process: first we pick $\sigma_T := \sigma\vert_{T}$ (distributed according to $\mu_{G[T]}$), and then pick $\sigma_{V_j} := \sigma \vert_{V_j}$ for each $V_j$ in a conditional way, so that $\sigma \sim \mu_G$ in total.

     Let $Z = (Z_1, \dots, Z_q)$ be the color class sizes of $\sigma_{T}$ and, for each $j \in [s]$, let $Q^j = (Q_1^j, \dots, Q_{q-1}^j)$ be the first $q-1$ color class sizes of $\sigma_{V_j}$. Define $R^j = (R_1^j, \dots, R_{q-1}^j)$ to be an independent copy of $Q^j$ (conditioning on the same $\sigma_T$).
     
    Note that $|\E[e^{-i{\langle}t, Q^{j}{\rangle}}]|^2 = \E[e^{-i{\langle}t, Q^{j} - R^{j}{\rangle}}]$ since $Q^j$ and $R^j$ are independent. For a fixed $t$, note that $\langle t, Q^j - R^j\rangle$ is symmetric around $0$ (as a random variable). Hence,
    \begin{align*}
          \E[e^{-i{\langle}t, Q^j - R^j{\rangle}}] &= \E[\cos (\langle t, Q^j - R^j \rangle)]  \nonumber \\
          &\leq 1 - \sum_{k = 1}^{q-1} \P[\langle t, Q^j - R^j \rangle = t_k] (1 - \cos (t_k))
    \end{align*} 
    For the last inequality, we would need that the $t_k$ are pairwise different, but we can always guarantee that by taking a small enough perturbation of the vector $t$.

    Let us now lower bound the probability $\P[{\langle}t, Q^j - R^j{\rangle} = t_k]$. We want to show it is at least a constant (which may depend on $\Delta$, $q$ and $\vl$). 
    \begin{claim}
        \[\P[{\langle}t, Q^j - R^j{\rangle} = t_k] = \Omega_{\Delta, q, \vl}(1)~~~ \forall k\in [q-1]\]
    \end{claim}
    \begin{claimproof}
        We may asssume $k=1$, since the proof is identical for any other coordinate. Label the neighbors of $v_j$ arbitrarily as $N(v_j) = \{w_1, \dots, w_d\}$, where $d := \deg(v_j)$. Let $\sigma^Q$ be the coloring of $V_j$ which gives the vector of color class sizes $Q^j$. Similarly, let $\sigma^R$ be the coloring associated to $R^j$. 

        Note that the probability that $\sigma^Q$ or $\sigma^R$ take a specific coloring as value is at least $\lambda_\text{min}^{\Delta + 1} / (q^{\Delta + 1} \lambda_\text{max}^{\Delta +1}) = \Omega_{\Delta, q, \vec \lambda}(1)$, so it is enough to show that there exists one such pair of colorings such that $\langle t, Q^j-R^j\rangle = t_1$. That is achieved by taking $\sigma^Q(v_j) = 1$, $\sigma^R(v_j) = q$, and $\sigma^Q(w_i) = \sigma^R(w_i) = c_i$ for all $i \in [d]$, where $c_i$ is a coloring not appearing in $N(w_i) \cap T$ and different from $1$ and $q$. One such color must exist for each $i \in [d]$, since $q \geq \max(2 \Delta, 3) \geq \Delta + 2$. Note that the constant we obtain does not depend on the conditioning on the coloring of $T$.
    \end{claimproof}

    For all t $\in [-\pi, \pi]$, by Taylor expansion, $\frac{1}{8}t^2 \leq 1 - \cos(t)$. Therefore, we can conclude that \[\E[e^{-i{\langle}t, Q^j - R^j{\rangle}}] \leq 1 - \frac 1 8 \sum_{k = 1}^{q-1} \P[\langle t, Q^j - R^j \rangle = t_k] t_k^2 \leq 1 - C\|t\|^2,\]
    for a certain $C = \Omega_{\Delta, q, \vec \lambda}(1)$. As mentioned before, this constant does not depend on the value of $\sigma_T$ we are conditioning on. Hence, 
    \[
        \left\vert \E[e^{-i \langle t, \vec X \rangle}] \right \vert \leq \max_{\sigma_T} \left \vert \E[e^{-i \langle t, \vec X \rangle} \, \vert \, \sigma_T] \right \vert = \max_{\sigma_T} \prod_{j = 1}^s \left \vert \E[e^{-i \langle t, Q^j \rangle}] \right \vert \leq \left( 1 - C \|t \|^2 \right)^{s/2} \leq \exp(-Cs \|t \|^2 / 2)
    \]
    Finally, since $s = \Omega(n)$, the claimed bound follows.
\end{proof}

\subsection{Proof of \autoref{thm:lclt}}\label{subsection: lclt together}
Before we put the ingredients from the previous sections together, we remind the reader of the standard inversion formula.

\begin{lemma}[(Inversion formula for an integer random vector)]\label{lemma: inv integer}
    For an integer-valued $d$-dimensional random vector $X$, its probability mass function $f_X(x)$ is given by the formula 
    \[ f_X(x) = \frac{1}{(2\pi)^d} \int_{[-\pi, \pi]^d} e^{-i{\langle}t,x{\rangle}} \phi_X(t) dt .\] 
\end{lemma}

Now we will prove the main result of this section. 

\lclt*

\begin{proof}
By Fourier inversion (\autoref{lemma: inv integer}), 
\[
\mathbb P(\vec X=\vec{n})
=
\frac{1}{(2\pi)^{q-1}}
\int_{[-\pi,\pi]^{q-1}}
e^{-i\langle t, \vec{n}\rangle}\varphi(t)\,dt.
\]

Let $A = \frac{R}{4 + 2R}$ and fix an $\varepsilon \in (0, 1/6)$. We will decompose
\[
[-\pi,\pi]^{q-1}
=
\mathcal R_1\cup\mathcal R_2
\]
where
\[
\mathcal R_1 :=\{t \in [-\pi, \pi]^{q-1} \, : \, \|t\|\le An^{-\frac{1}{2}+\eps}\}\quad \text{and} \quad
\mathcal R_2 := [-\pi,\pi]^{q-1}\setminus \mathcal R_1.
\]
 
By \autoref{lemma: characteristic upperbound}, for $t\in \mathcal R_2$,
\[
|\varphi(t)|
\le
\exp(-c n\|t\|^2)
\le
\exp(-c' n^{2\varepsilon}).
\]
Since $\mathrm{vol}([-\pi,\pi]^{q-1})=O(1)$, it follows that
\[
\int_{\mathcal R_2} |\varphi(t)|\,dt
\le
C \exp(-c n^{2\varepsilon})
=
o\!\left(n^{-(q-1)/2}\right).
\]
Consequently,
\[
\int_{[-\pi,\pi]^{q-1}}
\varphi(t)e^{-i\langle t,\vec{n}-\vec{\mu}\rangle}\,dt
=
\int_{\mathcal R_1}
\varphi(t)e^{-i\langle t,\vec{n}-\vec{\mu}\rangle}\,dt
+
o\!\left(n^{-(q-1)/2}\right).
\]

For the $t \in \mathcal R_1$, since $\|t\|_\infty \leq \|t\| \leq A$, using \autoref{corollary:log_phi_expansion_with_q-1_coordinates}, we have that 
\[\log\varphi(t)=i\langle t,\vec{\mu} \rangle
-\frac12 t^\top \Sigma \,t
+ O(n\|t\|^3),\]

so for any $\|t\|\le An^{-\frac{1}{2}+\eps}$
\[\varphi(t)
=
\exp\!\left(it^\top \vec{\mu}-\tfrac12 t^\top \Sigma t\right)\bigl(1+o(1)\bigr).\]
Therefore,
\[\int_{\mathcal R_1}
e^{-i\langle t, \vec{n}\rangle}\varphi(t)\,dt
=
(1+o(1))\int_{\mathcal R_1}
e^{-\frac12 t^\top \Sigma t - i\langle t,\vec{n}-\vec{\mu}\rangle}\,dt\]

Moreover, we know that $t^\top \Sigma t \geq \nu \|t\|^2$, where $\nu$ is the smallest eigenvalue of $\Sigma$, and from the proof of \autoref{lower bound det} we know that $\nu = \Omega(n)$. Therefore, 
\[
\Big \vert \int_{\mathbb R^{q-1}\setminus \mathcal R_1}
e^{-\frac12 t^\top \Sigma t- i\langle t,\vec{n}-\vec{\mu}\rangle}\,dt \Big \vert
=
o\!\left(n^{-(q-1)/2}\right),
\]
so we can extend the domain of integration to $\mathbb R^{q-1}$ (where it will be easier to integrate). Then,
\[
\int_{\mathcal R_1}
e^{-i \langle t, \vec n \rangle}\varphi(t)\,dt
=
\int_{\mathbb R^{q-1}}
e^{-\frac12 t^\top \Sigma t - i\langle t,\vec{n}-\vec{\mu}\rangle}\,dt
+
o\!\left(n^{-(q-1)/2}\right).
\]

Evaluating the Gaussian integral yields
\[
\int_{\mathbb R^{q-1}}
e^{-\frac12 t^\top \Sigma t - i\langle t,\vec{n}-\vec{\mu}\rangle}\,dt
=
(2\pi)^{-(q-1)/2}
(\det\Sigma)^{-1/2}
\exp\!\left(
-\tfrac12 (\vec{n}-\vec{\mu})^\top \Sigma^{-1} (\vec{n}-\vec{\mu})
\right),
\]

Thus, putting everything together, we get
\[
\mathbb P(\vec X=\vec{n})
=
(1+o(1))\frac{1}{(2\pi)^{(q-1)/2}\sqrt{\det\Sigma}}
\exp\!\left(
-\tfrac12 (\vec{n}-\vec{\mu})^\top\Sigma^{-1}(\vec{n}-\vec{\mu})
\right)
+
o\!\left(n^{-(q-1)/2}\right).
\]
\end{proof}

The LCLT result is enough for us to get the running time for our rejection sampling algorithm. We use a classical result by Vigoda \cite{vigoda1999improved}, which provides an algorithm for approximately sampling a uniform coloring with $q \geq 11\Delta/6$ colors in $O(n \log n)$ steps. 
The LCLT, together with bounds on the determinant of the covariance matrix and its eigenvalues (\autoref{det asymp}), implies that for $\vec x \in \mathbb{Z}_{\geq 0}^{q-1}$ with $|x_i - n/q| \leq c\sqrt{n}$ (where $c>0$ is a constant), $\P(\vec{X}=\vec x)=\Theta\!\left(n^{-(q-1)/2}\right)$. It is easy to show that after $O(n^{(q-1)/2} \log(1/\eps))$ iterations of rejection sampling, the probability of the algorithm to fail is at most $1-\eps$. Hence, running time of the algorithm is $O\!\left(n^{(q+1)/2}\, \log(n)\,\log(\frac{1}{\eps})^2\right)$. 

Note that in this section we assumed that we know $\vl$ such that $\E_{\vl}[\vec{X}] = \vec n$. For sampling equitable (and close to equitable) colorings, it is enough for us to take $\vl = \vec 1$. In the next section, we show how to sample colorings in which the color class sizes are much more skewed.

\section{Sampling skewed colorings}\label{section: skewed colorings}

Recall that we denote by $R$ the zero-freeness radius given by \autoref{corollary:zero-freeness-radius}, and we define the region $\mathcal D_R := \{\vec \lambda \in \mathbb C \, : \, \|\vec \lambda - \vec 1 \|_\infty < R\}$. By \autoref{remark:existence-branch-for-logZ}, we know that $F(\vec \lambda) := \log Z(\vec \lambda)$ can be defined analytically in $\mathcal D_R$.

The results from Section~\ref{section: LCLT} show that to sample colorings with a given vector of color class sizes $\vec n$ it suffices to find a $\vec \lambda$ in the zero-freeness region $\mathcal D_R$ such that
\[
    \|\E_{\vl}[\vec X] - \vec n\| = O(\sqrt n).
\]

Since finding such a $\vec \lambda$ for an arbitrary $\vec n$ is non-trivial, we discretize the search space and consider all candidate vectors of the form $\lambda_q = 1$,   
\[
    \lambda_i =  1 + \frac{1}{\sqrt n}  k_i, \quad \text{where }  k_i \in \{-\lfloor R \sqrt n \rfloor, \dots, \lfloor R \sqrt n \rfloor\}, \, i \in [q-1]
\]
and run the rejection sampling procedure for $O\!\left(n^{(q+1)/2}\, \log(n)\,\log(\frac{1}{\eps})^2\right)$ independently for each candidate $\vec \lambda$. We will show in the following Subsection \ref{subsection: Lipschitzness of the expectation map} that since our discretization is fine enough, at least  one of the $\vl$ values will give $\E_{\vl}[\vec X]$ within distance $\sqrt{n}$ of the desired coloring. There are $\Theta(n^{(q-1)/2})$ candidates, so the overall runtime will be $O\!\left(n^{q}\, \log(n)\,\log(\frac{1}{\eps})^2\right)$.
We show the surjectivity of the expected values on the given zero-freeness region in Subsection \ref{subsection: Surjectiveness of the expectation map}, which shows that for every $\vec{n}$ in our desired set of colorings, there exists $\vl$ which gives $\vec n$ in expectation. This helps us determine which colorings we can sample. Lastly, in Subsection \ref{subsection: Glauber} we show using a path coupling technique that Glauber dynamics mixes fast given $\vl$ from the zero-freeness region.

\subsection{Lipschitzness of the expectation map}\label{subsection: Lipschitzness of the expectation map}
The main statement proved in this section is that there exists a constant $C > 0$ such that for any $\vec \lambda, \vec \lambda' \in \mathcal D_R$,
    \[
        \|\E_{\vl}[\vec X] - \E_{\vl'}[\vec X] \| \leq Cn \|\vec \lambda - \vec \lambda' \|.
    \]

Let us first establish some notation. Scaling the vector $\vec \lambda$ by a positive factor has no effect on the probabilistic model, so we will use the notation $\vec \lambda := (\lambda_1, \dots, \lambda_{q-1})$ and assume that we have an implicit extra coordinate $\lambda_q = 1$. Let $\Psi : \mathbb R^{q-1} \longrightarrow \mathbb R^{q-1}$ be the endomorphism that associates each $\vec \lambda$ with the vector of expected color class sizes it induces:
\[
\Psi(\vec\lambda)
:=
\big(
\mathbb E_{\vec\lambda}[X_1],\dots,\mathbb E_{\vec\lambda}[X_{q-1}]
\big).
\]
We will use $\|\cdot\|$ to denote the $2$-norm, and the associated operator norm in the case of matrices.

\begin{lemma}\label{lemma:jacobian_as_covariance_times_diagonal}
    The Jacobian matrix $J(\vec \lambda) := D \Psi(\vec \lambda)$ satisfies
    \[
        J(\vec \lambda) = \cov_{\vl}(\vec X) \cdot D(\vec \lambda)^{-1},
    \]
    where $D(\vec \lambda) := \operatorname{diag}(\vec \lambda)$.
\end{lemma}
\begin{proof}
    By \autoref{lemma:expected_values_are_derivatives}, $\E_{\vec \lambda}[X_i] = \lambda_i \partial_i F(\vec \lambda)$, so we can write the entries of the Jacobian as
    \[
        J_{ij}(\vec \lambda) = \partial_j \E_{\vl}[X_i] = \lambda_i \partial_{ij}F(\vec \lambda) + \delta_{ij} \partial_i F(\vec \lambda).
    \]
    Comparing with \autoref{lemma:covariance_is_double_derivative}, we obtain that
    \[
        \lambda_j J_{ij}(\vec \lambda) = \cov_{\vec \lambda}(X_i, X_j),
    \]
    which is equivalent to
    \[
        J(\vec \lambda)  D(\vec \lambda) = \cov_{\vec \lambda}(\vec X)
    \]
\end{proof}

The formula from \autoref{lemma:jacobian_as_covariance_times_diagonal} implies a bound on the norm of the Jacobian:

\begin{lemma}\label{lemma:jacobian_norm_upperbound}
    There exists a constant $C = C(R, q, \Delta) > 0$ such that, for all $\vec \lambda \in \mathcal D_R$,
    \[
        \|J(\vec \lambda)\| \leq C n
    \]
\end{lemma}
\begin{proof}
    By \autoref{lemma:jacobian_as_covariance_times_diagonal}, $J(\vec \lambda) = \cov_{\vl}(\vec X) \cdot D(\vec \lambda)^{-1}$, which are both symmetric. Using the submultiplicativity of the operator norm and the fact that the norm of a real symmetric matrix is equal to its largest eigenvalue, we obtain that
    \[
        \| J(\vec \lambda) \| \leq \|\cov_{\vl}(\vec X)\| \cdot \|D(\vec \lambda)^{-1}\| \leq \frac{\tilde C n}{1 - R} 
    \]
    where the constant $\tilde C$ is the one given by \autoref{cor: up bound on det}, and $1-R$ is a lower bound on the smallest coordinate of $\vec \lambda$.
\end{proof}

\begin{lemma}\label{lemma:jacobian-lipschitzness}
    There exists a constant $C = C(R, q, \Delta) > 0$ such that, for all $\vec \lambda, \vec \lambda' \in \mathcal D_R$, 
    \[
        \| \Psi(\vec \lambda) - \Psi(\vec \lambda') \| \leq Cn \|\vec \lambda - \vec \lambda' \|
    \]
\end{lemma}
\begin{proof}
    Fix $\vec \lambda, \vec \lambda' \in \mathcal D$ and let $J_t := J(\vec \lambda' + t(\vec \lambda - \vec \lambda'))$ for $t \in [0,1]$. By the Fundamental Theorem of Calculus,
    \[
        \Psi(\vec \lambda) - \Psi(\vec \lambda') = \int_0^1 J_t \cdot (\vec \lambda - \vec \lambda') \, dt.
    \]
    Therefore,
    \[
        \|\Psi(\vec \lambda) - \Psi(\vec \lambda') \| \leq \int_0^1 \|J_t \cdot (\vec \lambda - \vec \lambda') \| \, dt \leq \int_0^1 \| J_t \| \cdot \| \vec \lambda - \vec \lambda' \| \, dt \leq Cn \| \vec \lambda - \vec \lambda' \|
    \]
    where $C$ is the constant given by \autoref{lemma:jacobian_norm_upperbound}.
\end{proof}

\subsection{Surjectiveness of the expectation map}\label{subsection: Surjectiveness of the expectation map}

We want to prove the following result:
\begin{lemma}\label{lemma:psi-surjective}
    There exists a constant $c = c(R, q, \Delta) > 0$ such that for any $\vec n$ satisfying 
    \[
        \left\|\vec n - \frac{n}{q}\vec 1 \right\|_\infty \leq cn,
    \]
    there exists a $\vec \lambda \in \mathcal D_R$ such that $\E_{\vl}[\vec X] = \vec n$.
\end{lemma}

For ease of analysis, we will reparametrize the model in terms of $\vec\theta\in\mathbb R^{q-1}$, where $\theta_i := \log \lambda_i$. As in the previous section, the last fugacity parameter is fixed to $\lambda_q = 1$. Then, within the zero-freeness region $\mathcal D := \{\vec \theta \in \mathbb R^{q-1} \, : \,\vert e^{\theta_i} - 1 \vert < R, \quad \forall i \in [q-1]\}$, we get the following analogues to \autoref{lemma:expected_values_are_derivatives} and \autoref{lemma:covariance_is_double_derivative}:
\begin{gather*}
    \nabla \log Z(\vec \theta) = \Psi(\vec \theta) \\
    \nabla^2\log Z(\vec \theta) = \cov_{\vec \theta}(\vec X)
\end{gather*}

We will now prove a few analytic considerations about the expectation map $\Psi : \mathbb R^{q-1} \longrightarrow \mathbb R^{q-1}$ defined by $\Psi(\vec \theta) := \E_{\vec \theta}[\vec X]$, which will be used to prove \autoref{lemma:psi-surjective}.

\begin{claim}\label{claim:psi-injective}
    There exists a constant $c_1 = c_1(R, q, \Delta)$ such that for every $\vec \theta, \vec \theta' \in \mathcal D$, 
    \[
        \| \Psi(\vec\theta) - \Psi(\vec \theta')\| \geq c_1 n \| \vec \theta - \vec \theta'\|.
    \]
    Consequently, $\Psi$ is injective within $\mathcal D$.
\end{claim}
\begin{proof}
By \autoref{lower bound det}, $D\Psi(\vec \theta) = \cov_{\vec \theta}(\vec X)$ has smallest eigenvalue $\Omega(n)$. Since it is symmetric, that means that there exists a constant $c_1$ such that
\[
    \langle D\Psi(\vec \theta) \cdot v, v \rangle \geq c_1 n \|v\|^2.
\]
Therefore, by the Fundamental Theorem of Calculus,
\begin{align*}
    \langle \Psi(\vec \theta) - \Psi(\vec \theta'),\, \vec \theta - \vec \theta' \rangle &= \langle \int_0^1 D \Psi\left(\vec \theta' + t(\vec \theta - \vec \theta')\right) \cdot (\vec \theta - \vec \theta') \, dt, \, \vec \theta - \vec \theta' \rangle \\
    &= \int_0^1 \langle  D \Psi\left(\vec \theta' + t(\vec \theta - \vec \theta')\right) \cdot (\vec \theta - \vec \theta') , \,  \vec \theta - \vec \theta' \rangle \, dt \\
    &\geq c_1 n \|\vec \theta - \vec \theta'\|^2
\end{align*}
By Cauchy-Schwarz, $\langle \Psi(\vec \theta) - \Psi(\vec \theta'), \, \vec \theta - \vec \theta' \rangle \leq \|\Psi(\vec \theta) - \Psi(\vec \theta') \| \cdot \|\vec\theta - \vec \theta' \|$, so we obtain
\[
    \| \Psi(\vec \theta) - \Psi(\vec \theta') \| \geq c_1 n \| \vec \theta - \vec \theta' \|,
\]
which implies that $\Psi$ is injective, as desired.
\end{proof}

\begin{claim}\label{claim:psi-open-map}
    $\Psi$ is an open map.
\end{claim}
\begin{proof}
    By \autoref{lower bound det}, the eigenvalues of $D\Psi(\vec \theta)$ are positive, so $D\Psi(\vec \theta)$ is invertible for any $\vec \theta \in \mathcal D$. Hence, by the Inverse Function Theorem\footnote{For instance, we can use the version from Theorem 9.24 of \cite{rudin1976principles}.}, $\Psi$ is a local diffeomorphism in $\mathcal D$. In particular\footnote{Being a local homeomorphism already suffices, as shown in \cite{lee2000introduction}.}, this implies that $\Psi(U)$ is open for any open $U \subseteq \mathcal D$.
\end{proof}

\begin{proof}[Proof of \autoref{lemma:psi-surjective}]

Note that, by taking $r := \log(1 + R)$, the ball $B(\vec 0, r) := \{\vec \theta \in \mathbb R^{q-1} \, : \, \| \vec \theta \| < r\}$ is contained within $\mathcal D$, as $\log (1 - R) < -\log(1 + R)$.

Consider the image of this ball by the expectation map. By \autoref{claim:psi-open-map}, $\Psi(B(\vec 0, r))$ is an open set, and we know that it contains $\Psi(\vec 0) = \frac{n}{q}\vec 1$. On the other hand, by \autoref{claim:psi-injective}, for any $\vec \theta \in \partial B(\vec 0, r)$ we have
\[
    \|\Psi(\vec \theta) - \Psi(\vec 0)\| \geq c_1 n \| \vec \theta \| = c_1 n r, 
\]
so $\Psi(\partial B(\vec 0, r))$ lies outside of $B(\Psi(\vec 0), c_1nr)$. Together with the fact that $\Psi(B(\vec 0, r))$ is open and $\Psi(B(\vec 0, r)) \cap B(\Psi(\vec 0), c_1nr) \neq \emptyset$, this implies that $B(\Psi(\vec 0), c_1 nr) \subseteq \Psi(B(\vec 0, r))$.

That means that any $\vec n$ with $\|\vec n - \frac{n}{q}\vec 1 \| < c_1 n r$ is contained in $\Psi(\mathcal D)$. Finally, taking $c := c_1r / \sqrt q$, we have that if $\|\vec n - \frac{n}{q}\vec 1 \|_\infty \leq cn$, then
\[
    \left\|\vec n - \frac{n}{q}\vec 1 \right\| \leq \sqrt {q-1} \left\| \vec n - \frac{n}{q}\vec 1 \right\|_\infty < c_1 n r
\]
and so $\vec n \in \Psi(\mathcal D)$, as desired.
\end{proof}

\subsection{Glauber dynamics}\label{subsection: Glauber}
We need to show that Glauber dynamics mixes in $O(n\log n \log(\frac{1}{\eps}))$ steps given $\vl$ in the zero-freeness region. We do this with a standard path coupling argument (for reference, one can see \cite{wilmer2009markov}). 

Let us first define Glauber dynamics for a given $\vl$. 
\begin{definition}[Glauber dynamics]\label{definition: mc}
Given a graph $G$ on $n$ vertices with maximum degree $\Delta$, a number of colors $q \in \mathbb Z^+$, a collection of weights $\lambda_c > 0$ for each color $c \in [q]$, and an initial coloring $x \in [q]^{V(G)}$, we define a Markov Chain $\{X_t\}_{t \geq 0}$ on the space of (not necessarily proper) $q$-colorings of the vertices of $G$, in which $X_0 := x$ and the transitions are given by the following procedure:
    \begin{enumerate}
        \item Choose a vertex $v$ uniformly at random. 
        \item Choose a color $c$ at random from the list of available colors $L_v := [q] \setminus X_t(N(v))$, so that color $i$ is chosen with probability $\frac{\lambda_i}{\sum_{j \in L_v} \lambda_j}$.
        \item Let $X_{t+1}(v) \gets c$, and let $X_{t+1}(w) \gets X_t(w)$ for all other $w \neq v$.  
    \end{enumerate}
\end{definition}

\begin{lemma}
    Let $0 < R < 1/10$ be a constant. Then, for any $\vec \lambda \in \mathbb R^q$ satisfying $\|\vec \lambda - \vec 1 \|_\infty < R$, the Glauber Dynamics with parameter $\vec \lambda$ has $t_\text{mix}(\varepsilon) = O(n \log n \log (\frac{1}{\varepsilon}))$.
\end{lemma}

\begin{proof}    
Given two colorings $X_0$ and $Y_0$ which differ only on one vertex $v_0$, we need to define a coupling of $X_1$ and $Y_1$, where each marginal behaves as the Glauber Dynamics, i.e. $X_1 \sim P(X_0,\,  \cdot)$ and $Y_1 \sim P(Y_0,\, \cdot)$. Let $c_X := X_0(v_0)$ and $c_Y := Y_0(v_0)$.
   
    The coupling we use is the following:
    \begin{enumerate}
        \item Pick a vertex $v$ uniformly at random (the same for both chains).
        \item If $v \not \in N(v_0)$, then we know that $L_{v, X_0} = L_{v, Y_0}$. Hence, we can couple the chains so we always pick the same color in both chains. \label{good moves}
        \item For $v \in N(v_0)$ and $c \not \in \{c_X, c_Y \}$, we couple choosing color $c$ in both chains, with probability
        $$ p_c := \frac{\lambda_c}{\max\left\{\sum\limits_{i \in L_{v, X_0}} \lambda_i, \, \sum\limits_{i \in L_{v, Y_0}} \lambda_i \right\} }$$
        Note that in one of the two chains we may still have some probability remaining of choosing color $c$.
        \item We finish by coupling arbitrarily the choices of $c$ from each chain that still had some probability left over, including the cases $c \in \{c_X, c_Y\}$.
    \end{enumerate}

    Note that all moves where $v \not \in N(v)\cup \{v_0$\} keep the distance between $X$ and $Y$ same. The moves where we choose $v = v_0$ are ``good'' because they decrease the distance between $X$ and $Y$ by 1 (note that $L_{v, X_0} = L_{v, Y_0}$, so the color of $v$ always changes). The moves where we choose a vertex $v \in N(v_0)$ are ``bad'', since the distance between $X$ and $Y$ may increase by 1, with probability at most $1 - \sum\limits_{c \notin \{c_X, x_Y\}}{p_c}$. 
    \\\\
    The probability of a good move is $1/n$, while the probability of a bad move can be upperbounded by \[\frac{\Delta}{n} \left( 1 - \sum_{c \notin \{c_X, c_Y\}} p_c \right) = \frac{\Delta}{n} \frac{\lambda_{c_Y} \cdot \mathbf{1}_{c_Y \in L_{v, X_0}}}{\sum\limits_{i \in L_{v, X_0}}\lambda_i} \leq \frac{\Delta}{n} \frac{1 + R}{(q-\Delta)(1-R)},\] where we have assumed without loss of generality that $\sum\limits_{i\in L_{v, X_0}} \lambda_i  \geq \sum\limits_{i \in L_{v, Y_0}} \lambda_i$.

    Therefore, the expected difference in Hamming distance after one step of the coupling is at most
    \begin{align*}
        \E[d(X_{1}, Y_{1})] \, &\leq \, \E[d(X_0, Y_0)] - \frac{1}{n} + \frac{\Delta}{n} \frac{1 + R}{(q - \Delta)(1-R)} \\
        &\leq 1 - \frac{1}{n} + \frac{1-c}{n} \\
        &= 1  -\frac{c}{n},
    \end{align*}
    where $c := 1/10$. For the last inequality, we have used the lower bound on $q$ together with the fact that for $c= 1/10$, we have$\frac{2R + c - Rc}{(1-R)(1-c)} <1 $. That implies that
    \begin{align*}
        q \geq 2\Delta +1 \geq \left( 2 + \frac{2R + c - Rc}{(1-R)(1-c)} \right) \Delta = \left(1 + \frac{(1+R)}{(1-R)(1-c)} \right) \Delta
    \end{align*}
    so
    \[
        \frac{(q - \Delta)(1-R)}{\Delta (1 + R)} \geq \frac{1}{1-c},
    \]
    as was needed for the inequality.

    Finally, as $1 - c/n \leq e^{-c/n}$, we apply Corollary 14.8 from \cite{wilmer2009markov} with $\alpha := c / n$, and obtain that
        $t_{\text{mix}}(1/4) 
        = O(n \log n )$. 
\end{proof}

\section*{Acknowledgments}

 AK supported in part by a Georgia Tech ARC-ACO student fellowship.  WP supported in part by NSF grant CCF-2309708.  XP supported in part by a travel fellowship from the MSCA-RISE-2020 project RandNET (no. 101007705) and the grant PID2023-147202NB-I00 funded by MICIU/AEI/10.13039/501100011033.

\end{document}